\documentclass[]{interact}

\usepackage[numbers,sort&compress]{natbib}
\bibpunct[, ]{[}{]}{,}{n}{,}{,}
\renewcommand\bibfont{\fontsize{10}{12}\selectfont}
\makeatletter
\def\NAT@def@citea{\def\@citea{\NAT@separator}}
\makeatother

\theoremstyle{plain}
\newtheorem{theorem}{Theorem}[section]
\newtheorem{lemma}[theorem]{Lemma}
\newtheorem{corollary}[theorem]{Corollary}
\newtheorem{proposition}[theorem]{Proposition}

\theoremstyle{definition}
\newtheorem{definition}[theorem]{Definition}
\newtheorem{example}[theorem]{Example}

\theoremstyle{remark}
\newtheorem{remark}{Remark}

\newcommand{\R}{\mathbb R}

\newcommand{\X}{\mathbb X}
\newcommand{\Y}{\mathbb Y}
\newcommand{\Uball}{{\mathbb B}}
\newcommand{\Usfer}{{\mathbb S}}

\newcommand{\dom}{{\rm dom}\, }
\newcommand{\graph}{{\rm graph}\,}

\newcommand{\nullv}{\mathbf{0}}
\newcommand{\Mat}{\mathbf{M}}    
\newcommand{\Rot}{\mathbf{S}\mathbf{O}}    

\newcommand{\bd}{{\rm bd}\, }
\newcommand{\inte}{{\rm int}\, }

\newcommand{\SVI}{\tt{SVI}}    
\newcommand{\CSVI}{\tt{CSVI}}    
\newcommand{\VOP}{\tt{VOP}}    
\newcommand{\Cmap}{\mathcal{C}}
\newcommand{\Fmap}{\mathcal{F}}
\newcommand{\Gmap}{\mathcal{G}}
\newcommand{\Hmap}{\mathcal{H}}
\newcommand{\Smap}{\mathcal{S}}
\newcommand{\Rmap}{\mathcal{R}}
\newcommand{\CSmap}{\widetilde{\mathcal{S}}}
\newcommand{\IEmap}{\mathcal{IE}}    
\newcommand{\val}{{\rm val}}    
\newcommand{\ief}{{\rm ie}}    
\newcommand{\VOPmap}{\mathcal{F}_{f,\Rmap}}

\newcommand{\impf}{{\rm s}}   
\newcommand{\mef}{\nu}   
\newcommand{\cimpf}{\widetilde{\rm s}}   
\newcommand{\cmef}{\zeta}   
\newcommand{\Carcond}{\mathfrak{C}}   
\newcommand{\sur}{{\rm sur}\,}

\newcommand{\ball}[2]{{\rm B}\left[#1; #2\right]}    

\newcommand{\dist}[2]{{\rm dist}\left(#1,#2\right)}     
\newcommand{\exc}[2]{{\rm exc}\left(#1;#2\right)}   
\newcommand{\haus}[2]{{\rm haus}\left(#1;#2\right)}    

\newcommand{\inc}[3]{{\rm inc}(#1;#2;#3)}   
\newcommand{\dec}[3]{{\rm dec}(#1;#2;#3)}   

\begin{document}


\title{On some global implicit function theorems for set-valued inclusions with applications
to parametric vector optimization}

\author{
\name{A. Uderzo\textsuperscript{a}\thanks{CONTACT A. Uderzo. Email: amos.uderzo@unimib.it}}
\affil{\textsuperscript{a} Department of Mathematics and its Applications,
University of Milano-Bicocca, Milano, Italy}
}

\maketitle

\begin{abstract}
The present paper deals with the perturbation analysis of set-valued
inclusion problems, a problem format whose relevance has recently emerged
in such contexts as
robust and vector optimization as well as in vector equilibrium theory.
The set-valued inclusions here considered are parameterized by variables
belonging to a topological space, with and without constraints.
By proper techniques of variational analysis, some qualitative global
implicit function theorems are established, which ensure global solvability
of these problems and continuous dependence on the
parameter of the related solutions. Applications to parametric vector
optimization are discussed, aimed at deriving sufficient conditions
for the existence of ideal efficient solutions that depend
continuously on the parameter perturbations.
\end{abstract}

\begin{keywords}
Set-valued inclusion problem, Caristi-type condition, implicit function,
parametric vector optimization, ideal efficient solution.
\end{keywords}

\section{Introduction}

Let $(P,\tau)$ be a topological space, let $(X,d)$ be a metric space,
and let $(\Y,\|\cdot\|)$ be a normed vector space.
Given two set-valued mappings $\Cmap:P\rightrightarrows\Y$ and
$\Fmap:P\times X\rightrightarrows\Y$, with the variable $p\in P$
playing the role of a parameter, consider the following class of
problems:
$$
  \hbox{find $x\in X$:\ } \qquad \Fmap(p,x)\subseteq\Cmap(p).  \leqno (\SVI_p)
$$
This is the parameterized version of a generalized equation type
referred to as set-valued inclusion (see \cite{Uder19,Uder21,Uder21b,Uder22}).
The relevance of such kind of problem format, that can not be
cast in traditional generalized equations, has emerged in various
topics of optimization theory. For instance,
in the context of robust optimization (see \cite{BenNem98,BeGhNe09}),
set-valued inclusions appear as a natural way of formalizing the robust
fulfilment of cone constraint systems, which are defined by data
affected by a crude knowledge of those uncertain elements often
arising in real-world models.
In the context of vector optimization, the notion of ideal and of
weak efficiency can be readily formulated in terms of set-valued
inclusions (see \cite{Uder23}), so the latter provide useful
insights into the study of efficiency conditions.
More generally, set-valued inclusions enable one to characterize
strong solutions to vector equilibrium problems, which
include vector variational inequalities and complementarity problems
(see \cite{Uder23b}).

Following a research line started in \cite{Uder21} and carried
on in \cite{Uder21b}, the present paper deals with the perturbation
analysis of parameterized set-valued inclusion problems. This
means the study of various properties of the solution mapping
$\Smap:P\rightrightarrows X$ associated with $(\SVI_p)$, i.e.
$$
  \Smap(p)=\{x\in X:\ \Fmap(p,x)\subseteq\Cmap(p)\},
$$
without relying on an explicit description of its values
(what would require to solve $(\SVI_p)$ for each $p\in P$).
Instead, such a study is based on directly handling the problem data,
in the full spirit of implicit function theorems.
More precisely, the investigations conducted in the present
paper focus on the solvability of $(\SVI_p)$ and the stability
of the solution set with respect to parameter perturbations.
Some properties connected with these themes within a perturbation analysis
have been investigated already in \cite{Uder21}. In comparison
with the researches exposed there, the main distinguishing features
of the present analysis are two. The parameter space here is a mere
topological space. In such a setting, the possible lack of a metric
space structure on $(P,\tau)$ does not permit to put the analysis
in the framework of well known theories developed in variational
analysis. Therefore, the perturbation analysis is forced to deal
with qualitative forms of stability of $\Smap$, instead of the
quantitative ones.
On the other hand, the emphasis of the analysis is here on
global properties of $\Smap$, instead of the local ones.
Nonlocal implicit function theorems for traditional equation
systems, under a topological space parametrization, have been
recently addressed also in \cite{ArZhMo23}, which inspired the
present research work.

It is well known that the perturbation analysis of various kinds
of problems is able to drive important developments in the deep
understanding of a problem in its basic (unperturbed) form.
In the specific case of variational analysis,
the perturbation approach to generalized
equations/variational systems had the effect of stimulating the
introduction of innovative ideas and successful techniques.

In order to adequate the original setting of parameterized
set-valued inclusions to a broad range of applications, the present
analysis consider also the constrained counterpart of $(\SVI_p)$,
whose statement is as follows:
given a further set-valued mapping $\Rmap:P\rightrightarrows X$,
$$
  \hbox{find }\ x\in\Rmap(p): \qquad \Fmap(p,x)\subseteq\Cmap(p).  \leqno (\CSVI_p)
$$
The solution mapping $\CSmap:P\rightrightarrows X$ associated with $(\CSVI_p)$
is consequently defined by
$$
  \CSmap(p)=\{x\in\Rmap(p):\ \Fmap(p,x)\subseteq\Cmap(p)\}.
$$

The synopsis of the paper is as follows. In Section \ref{Sect:2}
preliminary notions and facts, mainly concerning set-valued mappings
and their properties, are gathered, which are technically
needed for developing the approach here proposed.
Section \ref{Sect:3} contains the main findings of the paper.
Essentially, they are sufficient conditions for the existence of continuous functions
that are implicitly defined on the whole parameter space by
parameterized set-valued inclusions. These conditions, established
in two subsections, consider both
the problems $(\SVI_p)$ and $(\CSVI_p)$. Connections with related
existing results are also discussed.
A specific application of the emerging theory to the qualitative
analysis of ideal efficiency in vector optimization is presented
in Section \ref{Sect:4}.

The basic notations in use throughout the paper are standard.
$\R$ denotes the field of real numbers, $\R^n$ the space of
vectors with $n$ real components, $\R^n_+$ the nonnegative orthant
in $\R^n$.
Given a function $\varphi:X\longrightarrow\R\cup\{\pm\infty\}$,
$\dom\varphi=\varphi^{-1}(\R)$ denotes its domain and, if $\alpha\in\R$,
$[\varphi\le\alpha]=\{x\in X:\ \varphi(x)\le\alpha\}$ denotes its
sublevel set.
Whenever $(X,d)$ is a metric space and $S\subseteq X$, $\dist{x}{S}=
\inf_{z\in S}d(z,x)$ denotes that distance of a point $x\in X$ from
$S$, with the convention that $\dist{x}{\varnothing}=+\infty$.
Consistently, if $r\ge 0$, $\ball{S}{r}=[\dist{\cdot}{S}\le r]$
indicates the $r$-enlargement of $S$ with radius $r$. In particular, if
$S=\{x\}$, $\ball{x}{r}$ denotes the closed ball with center $x$
and radius $r$.
The symbol $\inte S$ and $\bd S$ indicate the topological interior
and the boundary of $S$, respectively.
Whenever $(\Y,\|\cdot\|)$ is a normed space, $\nullv$ stands for its
null vector. In this context, $\ball{\nullv}{1}$ and $\bd\ball{\nullv}{1}$
will be simply indicated by $\Uball$ and $\Usfer$, respectively.
Given a set-valued mapping $\Gmap:X\rightrightarrows Y$, $\dom\Gmap=
\{x\in X:\ \Gmap(x)\ne\varnothing\}$ and $\graph\Gmap=\{(x,y)\in X\times Y:
\ y\in\Gmap(x)\}$ denote the effective domain and the graph of $\Gmap$,
respectively.

The acronyms l.s.c. and u.s.c. stand for lower semicontinuous
and upper semicontinuous, respectively, for all their meanings.

The meaning of further symbols employed in subsequent sections
will be explained contextually to their introduction.

\vskip1cm


\section{Preliminaries} \label{Sect:2}

Throughout the paper the following standing assumptions
are maintained:

\begin{itemize}

\item[$(\mathfrak{A}_0)$] $(X,d)$ is metrically complete;

\item[$(\mathfrak{A}_1)$] $\Fmap(p,x)$ is nonempty and closed for every
$(p,x)\in P\times X$;

\item[$(\mathfrak{A}_2)$] $\{\nullv\}\ne\Cmap(p)\subsetneqq\Y$ is a closed,
convex cone for every $p\in P$.

\end{itemize}

Assumption $(\mathfrak{A}_0)$ seems to be hardly avoidable, whenever
existence issues are addressed via iterative schemes or variational
principles.
Assumption $(\mathfrak{A}_1)$ allows one to prevent the trivial situation
$\varnothing\subset\Cmap(p)$ and ensure general topological properties.
Assumption $(\mathfrak{A}_2)$ has to do with both the proposed
approach of analysis and subsequent applications.

Following a recent approach to the theory of implicit and
inverse functions and to the theory of fixed and coincidence points
(see \cite{Arut15,AruZhu20,ArZhMo23}), the existence of problem
solutions will be achieved by minimizing proper functions, which
are shown to satisfy a Caristi-type condition. Recall that
a function $\varphi:X\longrightarrow [0,+\infty]$ is said to satisfy
a Caristi-type condition if there exists a constant $\kappa>0$ such that
$$
  \forall x\in [\varphi>0]\ \exists\hat x\in X\backslash\{x\}:\
  \varphi(\hat x)+\kappa d(x,\hat x)\le\varphi(x). \leqno (\Carcond)
$$

\begin{remark}   \label{rem:Carcond}
It is well known that, whenever a function $\varphi:X\longrightarrow [0,+\infty]$,
which is l.s.c. on a complete metric space $(X,d)$ with $\dom\varphi\ne\varnothing$,
satisfies the condition $(\Carcond)$, then for every $x_0\in X$ there exists
$x_\kappa\in X$ such that
\begin{equation}    \label{in:BPVP}
  \varphi(x_\kappa)=0 \qquad\hbox { and }\qquad d(x_\kappa,x_0)
  \le  {\varphi(x_0)\over\kappa}.
\end{equation}
Such a statement is in fact an equivalent reformulation of the Bishop-Phelps/Ekeland variational
principle (see, for instance, \cite[Chapter I.2]{GraDug03} and \cite[Chapter 1.5]{Peno13}).
For the purposes of the present analysis, it is relevant to note that the inequality
in (\ref{in:BPVP}) entails
$$
  \dist{x_0}{[\varphi\le 0]}\le {\varphi(x_0)\over\kappa}.
$$
\end{remark}

As the values taken by $\Cmap$ are closed sets, it is convenient to represent
the solution mapping $\Smap$ associated with $(\SVI_p)$ by
$$
  \Smap(p)=[\mef(p,\cdot)\le 0],
$$
where $\mef:P\times X\longrightarrow [0,+\infty]$ is defined by
\begin{equation}     \label{eq:defmef}
  \mef(p,x)=\exc{\Fmap(p,x)}{\Cmap(p)}.
\end{equation}
Here $\exc{A}{B}=\sup_{a\in A}\dist{a}{B}$ stands for the (metric)
excess of a subset $A$ of a metric space beyond a subset $B$ of the same space.

Below, for a better readability of the present work, some known
facts related to the excess function are recalled, which will be
employed in the sequel. Let $\{\nullv\}\ne C\subsetneqq\Y$ be a closed, convex
cone. Then
\begin{itemize}

\item[$(\mathfrak{P}_1)$] for every $S\subseteq\Y$ it holds $\exc{S+C}{C}=
\exc{S}{C}$;

\item[$(\mathfrak{P}_2)$] for every $r>0$ and $S\subseteq\Y$ such that $\exc{S}{C}>0$,
it holds $\exc{\ball{S}{r}}{C}=\exc{S}{C}+r$

\end{itemize}
(see \cite[Remark 2.1(iv)]{Uder19} and \cite[Lemma 2.2]{Uder19}, respectively).

As it is clear from Remark \ref{rem:Carcond},
to develop the subsequent analysis semicontinuity properties of $\mef$
are needed. It is well known that semicontinuity properties of such functions
as $\mef$ can be derived by corresponding semicontinuity properties
of $\Fmap$. The next remark points out some useful implications.

\begin{remark}    \label{rem:scontmef}
Let $(\Y,\|\cdot\|)$ be a normed vector space and let $C\subseteq\Y$ a closed, convex
cone.

(i) If a set-valued mapping $\Gmap:X\rightrightarrows\Y$ defined on a metric space is
l.s.c. at $x_0\in\dom\Gmap$, then the function $x\mapsto\exc{\Gmap(x)}{C}$ is l.s.c.
at $x_0$ (see \cite[Lemma 2.3(i)]{Uder19}).

(ii) If a set-valued mapping $\Gmap:P\rightrightarrows\Y$ defined on a topological space is
Hausdorff $C$-u.s.c. at $p_0\in\dom\Gmap$, then the function $p\mapsto\exc{\Gmap(p)}{C}$
is u.s.c. at $p_0$.

To show this, it suffices to adapt the proof of \cite[Lemma 3.2(ii)]{Uder19} to the
more general setting of topological spaces. So, take an arbitrary $\epsilon>0$. Since
$\Gmap$ is Hausdorff $C$-u.s.c. at $p_0$, there exists a neighbourhood $U_\epsilon$
of $p_0$ such that
$$
   \Gmap(p)\subseteq \ball{\Gmap(p_0)+C}{\epsilon/2},\quad\forall p\in U_\epsilon.
$$
If it is $\Gmap(p_0)\subseteq C$, then by taking into account that for any
$\varnothing\ne S\subseteq\Y$ and $r>0$ the inclusion
$$
   \ball{S+C}{r}\subseteq\bigcap_{t>r}\left(\ball{S}{t}+C\right)
$$
holds and by using property $(\mathfrak{P}_1)$, one finds
\begin{eqnarray*}
  \exc{\Gmap(p)}{C} &\le& \sup_{y\in\ball{\Gmap(p_0)+C}{\epsilon/2}}\dist{y}{C}\le
    \sup_{y\in\ball{\Gmap(p_0)}{\epsilon}+C}\dist{y}{C} \\
    &=& \exc{\ball{\Gmap(p_0)}{\epsilon}+C}{C}=\exc{\ball{\Gmap(p_0)}{\epsilon}}{C}   \\
    &\le&   \sup_{y\in\ball{C}{\epsilon}}\dist{y}{C}=\epsilon,
   \quad\forall p\in U_\epsilon.
\end{eqnarray*}
Otherwise, on account of properties $(\mathfrak{P}_1)$
and $(\mathfrak{P}_2)$, where the latter can now be applied because $\exc{\Gmap(p_0)}{C}>0$
and hence it is $\exc{\Gmap(p_0)+C}{C}>0$, it holds
\begin{eqnarray*}
   \exc{\Gmap(p)}{C} &\le &\exc{\ball{\Gmap(p_0)+C}{\epsilon}}{C}=
   \exc{\Gmap(p_0)+C}{C}+\epsilon   \\
   &=& \exc{\Gmap(p_0)}{C}+\epsilon, \quad\forall p\in U_\epsilon.
\end{eqnarray*}

(iii) If a set-valued mapping $\Gmap:X\rightrightarrows\Y$ defined on a metric space is
Lipschitz continuous on $X$ with constant $\ell$, i.e. there exists $\ell>0$ such that
$$
  \haus{\Gmap(x_1)}{\Gmap(x_2)}\le\ell d(x_1,x_2),\quad\forall x_1,\, x_2\in X,
$$
where $\haus{A}{B}=\max\{\exc{A}{B},\, \exc{B}{A}\}$,
then the function $p\mapsto\exc{\Gmap(p)}{C}$ is Lipschitz continuous on $X$ with
the same constant.
This fact comes as an immediate consequence of the inequality
$\exc{A}{C}\le \exc{A}{B}+\exc{B}{C}$, which is valid for any triple
$A$, $B$, and $C$ of nonempty subsets of a metric space.
\end{remark}

\begin{remark}   \label{rem:uscdistp}
If a set-valued mapping $\Gmap:P\rightrightarrows X$, defined on a topological
space with values in a metric space, is l.s.c. at $p_0\in P$, then for every $x\in X$
the function $p\mapsto\dist{x}{\Gmap(p)}$ is u.s.c. at the same point
(see, for instance, \cite[Proposition 1.45]{Thib23}).
\end{remark}


As a next tool of analysis, a property for set-valued mappings, which are
defined in a metric space and take values in a normed vector space
partially ordered by a cone, is recalled. Such a property has revealed to play
a crucial role in the study of the solvability of set-valued
inclusion problems (see \cite{Uder19,Uder22}).

\begin{definition}[{\bf $C$-increase property}]    \label{def:Cincrpropdef}
Let $\Gmap:X\rightrightarrows\Y$ be a set-valued
mapping, let $\{\nullv\}\ne C\subsetneqq\Y$ be a closed, convex cone, and let
$x_0\in X$. $\Gmap$ is said to be {\it (metrically) $C$-increasing} at
$x_0\in\dom\Gmap$ if there
exist $\alpha>1$ and $\delta>0$ such that
\begin{equation}    \label{def:Cincrprop}
   \forall r\in (0,\delta]\ \exists u\in\ball{x_0}{r}:\
   \ball{\Gmap(u)}{\alpha r}\subseteq \ball{\Gmap(x_0)+C}{r}.
\end{equation}
The value
\begin{equation}    \label{def:Cincrbxbo}
   \inc{\Gmap}{C}{x_0}=\sup\{\alpha>1:\ \exists\delta>0 \hbox{ for which (\ref{def:Cincrprop})
   holds\,}\}
\end{equation}
is called {\it exact bound of $C$-increase} of $\Gmap$ at $x_0$. If $\Gmap$
is $C$-increasing at each $x_0\in X$, then it is said to be $C$-increasing on $X$.
\end{definition}

\begin{remark}   \label{rem:unexCincrprop}
In view of subsequent arguments, it is useful to observe that, whenever
a set-valued mapping $\Gmap$ is $C$-increasing at $x_0$ while it is
$\Gmap(x_0)\not\subseteq C$, then, if $\alpha$ and $\delta$ are as in Definition
\ref{def:Cincrpropdef}, for every $r\in (0,\delta]$ the inclusion in
(\ref{def:Cincrprop}) must be necessarily satisfied by $u\in\ball{x_0}{r}
\backslash\{x_0\}$. Indeed, if it were $u=x_0$, one would obtain the inclusion
$$
  \ball{\Gmap(x_0)}{\alpha r}\subseteq \ball{\Gmap(x_0)+C}{r},
$$
which, by taking into account that $\exc{\Gmap(x_0)}{C}>0$, can never hold
true. This because, by the properties $(\mathfrak{P}_1)$ and $(\mathfrak{P}_2)$,
it would yield
\begin{eqnarray*}
  \exc{\Gmap(x_0)}{C}+\alpha r &=& \exc{\ball{\Gmap(x_0)}{\alpha r}}{C} \\
   &\le & \exc{\ball{\Gmap(x_0)+C}{r}}{C}=\exc{\Gmap(x_0)+C}{C}+r \\
   &=& \exc{\Gmap(x_0)}{C}+r,
\end{eqnarray*}
which implies $\alpha\le 1$, in contradiction with $\alpha>1$, as required in
Definition \ref{def:Cincrpropdef}.
\end{remark}

\begin{example}(Rescaled rotations of $\R^n$)    \label{ex:rescarot}
For any $n\ge 2$, let $X=\Y=\R^n$ be equipped with its usual Euclidean space
structure, and let $C=\R^n_+$. Let $(\Rot(n),\circ)$ denote the group of all rotations
(a.k.a. special orthogonal group) of $\R^n$. For every real $\lambda>n$, the
rescaled rotation mapping $\lambda O:\R^n\longrightarrow\R^n$, with $O\in\Rot(n)$,
is $\R^n_+$-increasing at each point $x\in\R^n$, with
\begin{equation}\label{in:incrrot}
  \inc{\lambda O}{\R^n_+}{x}\ge\sqrt{n},\quad\forall x\in\R^n.
\end{equation}
To see this, recall that, according to \cite[Example 3.2]{Uder19},
whenever a linear mapping $\Lambda:\R^n\longrightarrow\R^n$ satisfies the condition
\begin{equation}     \label{in:surincrlinear}
   \sur\Lambda=\inf_{\|u\|=1}\|\Lambda^\top u\|=\dist{\nullv}{\Lambda^\top\Usfer}>n,
\end{equation}
then $\Lambda$ turns out to be metrically $\R^n_+$-increasing at each point $x$ of
$\R^n$, with $\inc{\Lambda}{\R^n_+}{x}\ge\sqrt{n}$.
Now, since it is $O^\top=O^{-1}\in\Rot(n)$ ($O^\top$ standing for the transposed of $O$)
and $O\Usfer=\Usfer$ for every $O\in\Rot(n)$,
one obtains
$$
  \sur(\lambda O)=\dist{\nullv}{(\lambda O)^\top\Usfer}=\dist{\nullv}{\lambda O^{-1}\Usfer}=
  \dist{\nullv}{\lambda\Usfer}=\lambda,
$$
which says that condition (\ref{in:surincrlinear}) is fulfilled inasmuch as $\lambda>n$.
The reader should notice that, more generally, for a linear mapping $\Lambda:\R^n\longrightarrow\R^m$
the value $\sur\Lambda$ provides a representation of the exact bound of
open covering of $\Lambda$ (see \cite[Corollary 1.58]{Mord06}).
For a comprehensive discussion of this
property (a.k.a. openness at a linear rate) and its crucial role in
variational analysis, the reader is referred to \cite{Mord06,Peno13,Thib23}).
Consistently, by dealing directly with Definition \ref{def:Cincrpropdef}  it is
possible to obtain for this class of mappings a sharper estimate of
the exact bound of $\R^n_+$-increase. Indeed, in the case $n=2$, given $\theta\in\R$,
consider
$$
 O_\theta=\left(\begin{array}{cc}
            \cos\theta & -\sin\theta \\
            \sin\theta & \cos\theta
          \end{array}\right)\in\Rot(2),\quad\forall \theta\in\R.
$$
Fixed an arbitrary $r>0$, since it holds
$$
   O_\theta\circ O_\eta=O_{\theta+\eta},\quad\forall \theta,\, \eta\in\R,
$$
then by setting
\begin{equation}     \label{eqin:defumetincpro}
  u=O_{{\pi\over 4}-\theta}\binom{r}{0}\in r\Usfer\subseteq r\Uball
\end{equation}
one finds
$$
  (\lambda O_\theta) u=\lambda O_\theta\circ O_{{\pi\over 4}-\theta}\binom{r}{0}=
  {\lambda r\over\sqrt{2}}\binom{1}{1}.
$$
Therefore, one obtains
\begin{eqnarray*}
  \ball{(\lambda O_\theta) u}{\left({\lambda\over\sqrt{2}}+1\right)r} &=&
  \ball{{\lambda r\over\sqrt{2}}\binom{1}{1}}{\left({\lambda\over\sqrt{2}}+1\right)r}
   \subseteq  \ball{\R^2_+}{r}= \\
   &=& \ball{(\lambda O_\theta)\nullv+\R^2_+}{r},\quad\forall r>0.
\end{eqnarray*}
The above inclusion shows that $\lambda O_\theta$ is $\R^2_+$-increasing
at $\nullv$ with $\inc{\lambda O_\theta}{\R^2_+}{\nullv}\ge {\lambda\over\sqrt{2}}+1$.
If arguing by means of level sets and symmetry, it is possible to deduce
that the point $u$, defined as in (\ref{eqin:defumetincpro}), solves
the problem
$$
   \max_{x\in r\Uball}\dist{\lambda O_\theta x}{\bd\R^2_+}.
$$
Equivalently, its image $\lambda O_\theta u$ is the farthest point
of $\lambda r\Uball$ from $\bd\R^2_+$. This fact leads to establish
the equality
$$
  \inc{\lambda O_\theta}{\R^2_+}{\nullv}={\lambda\over\sqrt{2}}+1.
$$
By virtue of the linearity of the mapping represented by $\lambda O_\theta$,
it is possible to conclude that for any $\theta\in\R$ it holds
\begin{equation}     \label{in:estincrlOtheta}
   \inc{\lambda O_\theta}{\R^2_+}{x}={\lambda\over\sqrt{2}}+1,
   \quad\forall x\in\R^2.
\end{equation}
\end{example}

\begin{example}
Let $X=\R$ and $\Y=\R^2$ be equipped with their usual Euclidean space structure,
and let $C=\R^2_+$. Define $\Gmap:\R\rightrightarrows\R^2$ by setting
$$
  \Gmap(x)=\{y=(y_1,y_2)\in\R^2:\ \min\{y_1,\, y_2\}=x\}.
$$
It is readily seen that the values taken by $\Gmap$ are closed but unbounded
subsets of $\R^2$. Fix arbitrary $x\in\R$ and $r>0$. Then, by choosing
$u=x+r\in\ball{x}{r}$, one finds
\begin{eqnarray*}
  \ball{\Gmap(u)}{2r} &=& \ball{(x+r)\binom{1}{1}+\bd\R^2_+}{2r} \\
   &\subseteq & \ball{x\binom{1}{1}+\R^2_+}{r}=\ball{\Gmap(x)+\R^2_+}{r}.
\end{eqnarray*}
The above inclusion enables one to assert that $\Gmap$ is $\R^2_+$-increasing
at every $x\in\dom\Gmap=\R$, with $\inc{\Gmap}{\R^2_+}{x}\ge 2$.
\end{example}

As remarked in \cite{Uder19}, many other examples of $C$-increasing mappings
can be built by exploiting the persistence of the $C$-increase property under small
additive perturbations, which is stated in the proposition below.

\begin{proposition}    \label{pro:Cincrprosta}
Let $\Gmap:X\rightrightarrows\R^m$ and $\Hmap:X\rightrightarrows\R^m$ be
set-valued mappings and let $C\subseteq\R^m$ be a closed, convex cone.
Suppose that:
\begin{itemize}

\item[(i)] $\Gmap$ is metrically $C$-increasing at $x_0\in X$;

\item[(ii)] $\Gmap+\Hmap$ is closed valued;

\item[(iii)] $\Hmap$ is Lipschitz continuous on $X$ with a
constant $\ell$ satisfying the condition
$$
  \ell<1-{1\over\inc{\Gmap}{C}{x_0}}.
$$
\end{itemize}
Then, the set-valued mapping $\Gmap+\Hmap:X\rightrightarrows\R^m$
is metrically $C$-increasing at $x_0$, with
$$
   \inc{\Gmap+\Hmap}{C}{x_0}\ge (1-\ell)\inc{\Gmap}{C}{x_0}.
$$
\end{proposition}

For details on its proof the reader is referred to \cite[Proposition 3.6]{Uder19}.

\vskip1cm


\section{Nonlocal implicit function theorems} \label{Sect:3}

The main findings of the paper are exposed in the two subsequent
sections, considering problems $(\SVI_p)$ and $(\CSVI_p)$ separately.
In both these subsections, the key property to achieve these results
is the increasing property of $\Fmap$ with respect to $\Cmap$.

\subsection{The unconstrained case}

The first step towards a global implicit function theorem for
problem $(\SVI_p)$ consists in ensuring global solution existence
in the presence of parameter perturbation, along with a proper
error bound, what is done in the next proposition by means of
the following constant, which can be built by the problem data:

\begin{equation}    \label{in:incFmapCX}
 \alpha_{\Fmap,\Cmap}=\inf\bigg\{\inc{\Fmap(p,\cdot)}{\Cmap(p)}{x}:\
 (p,x)\in P\times X,\ \Fmap(p,x)\not\subseteq\Cmap(p)\bigg\}.
\end{equation}

\begin{proposition}[{\bf Global solvability}]   \label{pro:glosolv}
With reference to a parameterized family of set-valued inclusion
problems $(\SVI_p)$, suppose that:

\begin{itemize}

\item[(i)] for every $p\in P$ there exists $\hat x_p\in X$ such that
$\exc{\Fmap(p,\hat x_p)}{\Cmap(p)}<+\infty$;

\item[(ii)] $\Fmap(p,\cdot):X\rightrightarrows\Y$ is l.s.c. on $X$,
for every $p\in P$;

\item[(iii)] it holds $\alpha_{\Fmap,\Cmap}>1$.

\end{itemize}
Then, for any $\alpha\in (1,\alpha_{\Fmap,\Cmap})$ and any $x_0\in X$ it holds
\begin{equation}   \label{in:glosolvthesis}
  \Smap(p)\ne\varnothing \quad\hbox{ and }\quad\dist{x_0}{\Smap(p)}\le
  {\exc{\Fmap(p,x_0)}{\Cmap(p)}\over \alpha-1},\quad\forall p\in P.
\end{equation}
\end{proposition}

\begin{proof}
Fix $x_0\in X$ and $\alpha\in (1,\alpha_{\Fmap,\Cmap})$.
Then fix an arbitrary $p\in P$. If $x_0\in\Smap(p)$, all the assertions in
(\ref{in:glosolvthesis}) are evidently true. So, assume that $x_0
\not\in\Smap(p)$ and consider the merit function associated with $p$, namely
$\mef_p:X\longrightarrow [0,+\infty]$, defined by
\begin{equation}    \label{eq:defmefp}
  \mef_p(x)=\exc{\Fmap(p,x)}{\Cmap(p)}.
\end{equation}
Notice that by virtue of hypothesis (i) it is $\dom\mef_p\ne\varnothing$,
inasmuch as $\mef_p(\hat x_p)<+\infty$. Moreover, according with what was recalled
in Remark \ref{rem:scontmef}(i), $\mef_p$ is l.s.c. on $X$ owing to hypothesis (ii).
By using hypothesis (iii) it is possible to show that $\mef_p$ satisfies
the Caristi-type condition ($\Carcond$), with $\kappa=\alpha-1$. Indeed,
take an arbitrary $x\in X$ such that $\mef_p(x)>0$ (including
the case $\mef_p(x)=+\infty$), so $\Fmap(p,x)\not\subseteq\Cmap(p)$.
As $\mef_p$ is, in particular,
l.s.c. at $x$, there exists $\delta_*>0$ such that
\begin{equation}   \label{in:mefpposnearx}
  \mef_p(z)>0,\quad\forall z\in \ball{x}{\delta_*}.
\end{equation}
Since $\Fmap(p,\cdot)$ is $\Cmap(p)$-increasing at $x$ with
$\inc{\Fmap(p,\cdot)}{\Cmap(p)}{x}>\alpha$, according to
(\ref{def:Cincrbxbo}) there must exist $\delta_\alpha>0$ such that
for every $r\in (0,\delta_\alpha]$ there is $u\in\ball{x}{r}$ such that
\begin{equation}   \label{in:CpincrproFp}
   \ball{\Fmap(p,u)}{\alpha r}\subseteq \ball{\Fmap(p,x)+\Cmap(p)}{r}.
\end{equation}
Thus, by choosing $0<r<\min\left\{\delta_*,\delta_\alpha\right\}$,
one can assert that, in the light of the fact that $\Fmap(p,x)\not\subseteq\Cmap(p)$
and Remark \ref{rem:unexCincrprop},
it must be $u\in X\backslash\{x\}$. Furthermore,
because of inequality (\ref{in:mefpposnearx}) it is $\mef_p(u)>0$.
Therefore, by exploiting property $(\mathfrak{P}_1)$, inclusion (\ref{in:CpincrproFp}),
and property $(\mathfrak{P}_2)$, one obtains
\begin{eqnarray*}
  \mef_p(u) &= & \exc{\Fmap(p,u)}{\Cmap(p)} =
  \exc{\ball{\Fmap(p,u)}{\alpha r}}{\Cmap(p)}-\alpha r  \\
  &\le & \exc{\ball{\Fmap(p,x)+\Cmap(p)}{r}}{\Cmap(p)}-\alpha r \\
  &= & \exc{\Fmap(p,x)+\Cmap(p)}{\Cmap(p)}+r-\alpha r
  = \exc{\Fmap(p,x)}{\Cmap(p)}-(\alpha-1)r,
\end{eqnarray*}
whence it follows
$$
  \mef_p(u)+(\alpha-1)r\le\mef_p(x).
$$
Since it is $d(u,x)\le r$, the last inequality entails
$$
  \mef_p(u)+(\alpha-1)d(u,x)\le\mef_p(x),
$$
meaning that condition ($\Carcond$) is actually satisfied with
$\kappa=\alpha-1$ and $\hat x=u$.
According to Remark \ref{rem:Carcond} and assumption $(\mathfrak{A}_0)$,
there must exist $x_p\in X$ such that
$$
  \mef_p(x_p)=0 \qquad\hbox { and }\qquad d(x_p,x_0)\le
  {\mef_p(x_0)\over\alpha-1}.
$$
This implies that $x_p\in\Smap(p)\ne\varnothing$ and consequently
the validity of the inequality in (\ref{in:glosolvthesis}),
thereby completing the proof by arbitrariness of $p\in P$.
\end{proof}

\begin{remark}
It should be noticed that hypothesis (i) of Proposition \ref{pro:glosolv}
is valid whenever $\Fmap(p,\cdot)$ takes at least one $\Cmap(p)$-bounded
value, for every $p\in P$, i.e.
$$
  \forall p\in P:\ \exists\hat x_p\in X:\ \Fmap(p,\hat x_p)\backslash
  \Cmap(p)\ \hbox{ is a bounded set in } \Y.
$$
\end{remark}

Proposition \ref{pro:glosolv} provides sufficient conditions for
the existence of solutions to $(\SVI_p)$ under changes of the parameter $p$ all over $P$,
stating equivalently that $\dom\Smap=P$. Besides, the solution existence is complemented with
an error bound, which, in synergy with some further assumptions, enables one
to restore already established results about the quantitative solution stability
of $(\SVI_p)$, in the particular case when $\Cmap$ is constant
(see \cite[Theorem 3.1, Theorem 3.3, Theorem 3.5]{Uder21}).
For instance, if one supposes, in addition to the hypotheses of Proposition
\ref{pro:glosolv}, that:
\begin{itemize}

\item[(iv)] $(\bar p,\bar x)\in\graph\Smap$;

\item[(v)] $\Cmap$ is constant, i.e. $\Cmap(p)=C$, for every $p\in  P$;

\item[(vi)] $(P,\tau)$ is metrizable by a metric $d$ and $\Fmap(\cdot,\bar x):
P\rightrightarrows\Y$ is Lipschitz u.s.c. at $\bar p$, i.e. there exist positive
real $\bar\ell$ and $\bar\delta$ such that
$$
  \Fmap(p,\bar x)\subseteq\ball{\Fmap(\bar p,\bar x)}{\bar\ell d(p,\bar p)},
  \quad\forall p\in\ball{\bar p}{\bar\delta},
$$
\end{itemize}
then $\Smap$ turns out to be Lipschitz l.s.c. at $\bar p$ (as established in
\cite[Theorem 3.1]{Uder21}), i.e. there exist positive
real $\underline{\ell}$ and $\underline{\delta}$ such that
\begin{equation}     \label{ne:SmapLiplsc}
   \Smap(p)\cap\ball{\bar x}{\underline{\ell}d(p,\bar p)}\ne\varnothing,
   \quad\forall  p\in\ball{\bar p}{\underline{\delta}}.
\end{equation}
To see this in detail, observe that according to \cite[Lemma 2.4(ii)]{Uder21},
due to hypothesis (vi) the function
$\mef_{\bar x}:P\longrightarrow[0,+\infty]$, defined by
$$
   \mef_{\bar x}(p)=\exc{\Fmap(p,\bar x)}{C},
$$
is calm from above at $\bar p$, i.e. there exist positive $\gamma$ and $\delta_\gamma$
such that
\begin{equation}    \label{in:mefcalmab}
  \mef_{\bar x}(p)=\mef_{\bar x}(p)-\mef_{\bar x}(\bar p)\le\gamma d(p,\bar p),
  \quad\forall p\in\ball{\bar p}{\delta_\gamma}.
\end{equation}
Then, by applying Proposition \ref{pro:glosolv} with $x_0=\bar x$,
from the inequalities in (\ref{in:glosolvthesis}) and (\ref{in:mefcalmab}),
one obtains
$$
  \dist{\bar x}{\Smap(p)}\le{\exc{\Fmap(p,\bar x)}{C}\over \alpha-1}\le
  {\gamma\over \alpha-1}d(p,\bar p),
  \quad\forall p\in\ball{\bar p}{\delta_\gamma},
$$
which gives (\ref{ne:SmapLiplsc}) with $\underline{\delta}=\delta_\gamma$
and $\underline{\ell}>\gamma/(\alpha-1)$.
Of course, the condition required by hypothesis (iii), involving the
constant $\alpha_{\Fmap,\Cmap}$, is stronger than the corresponding
assumption (hypothesis (iv) in \cite[Theorem 3.1]{Uder21}), which is of local
nature. On the other hand, it is by virtue of such a hypothesis that
Proposition \ref{pro:glosolv} ensures global solvability, in contrast to
\cite[Theorem 3.1]{Uder21}, which can only ensure local solvability.
Similar comments can be repeated for \cite[Theorem 3.3, Theorem 3.5]{Uder21},
by proper replacements of the quantitative Lipschitz semicontinuity assumption on
the term $\Fmap$.

It is important to remark also that the increase property assumption
refers to $\Fmap$ as a multifunction of the variable $x$ only.

The next step consists in obtaining a continuous implicit function
from the mere solvability of $(\SVI_p)$. Such a step will be carried out
in the particular case when $\Cmap$ takes a constant value $C\subset\Y$.
A property useful to pursue this goal has to do with concavity for
set-valued mappings. Consequently, throughout the rest of this subsection
the metric space $(X,d)$ is specialized to be a normed vector space,
denoted by $(\X,\|\cdot\|)$.

Recall that, given a cone $C\subseteq\Y$, a set-valued mapping
$\Gmap:\X\rightrightarrows\Y$ is said to be
$C$-concave on a convex subset $S\subseteq\X$ if
$$
  \Gmap(tx_1+(1-t)x_2)\subseteq t\Gmap(x_1)+(1-t)\Gmap(x_2)+C,
  \quad\forall x_1,\, x_2\in S,\ \forall t\in [0,1].
$$

\begin{example}[Compactly generated fans]     \label{ex:comgenfan}
Let $(\Mat_{m\times n}(\R),\|\cdot\|)$ denote the vector space
of all the $m\times n$ matrices with real entries, equipped
with the operator norm, and let $L$ be a nonempty, convex and
compact subset of $\Mat_{m\times n}(\R)$. Consider the set-valued
mapping $\Hmap_L:\R^n\rightrightarrows\R^m$ defined by
$$
  \Hmap_L(x)=\{\Lambda x:\ \Lambda\in L\}.
$$
Such kind of set-valued mapping, often referred to as a fan, takes nonempty, compact
and convex values. As readily seen, $\Hmap_L$ is positively homogeneous and satisfies
the inclusion
$$
   \Hmap_L(tx_1+(1-t)x_2)\subseteq t\Hmap_L(x_1)+(1-t)\Hmap_L(x_2),
  \quad\forall x_1,\, x_2\in\R^n,\ \forall t\in [0,1].
$$
Thus it turns out to be $C$-concave on $\R^n$, for every cone $C\subseteq\R^m$.
As $L$ is, in particular, a bounded subset of $\Mat_{m\times n}(\R)$,
$\Hmap_L$ can be shown to be Lipschitz continuous on $\R^n$, with the
following Lipschitz constant
\begin{equation}    \label{eq:fannorm}
   \ell_{\Hmap_L}=\max\{\|\Lambda\|: \Lambda\in L\}
\end{equation}
(see, for instance, \cite[Remark 2.14(iii)]{Uder22}).
\end{example}

Given any $p_0\in P$, the next lemma collects some properties
of the solution set $\Smap(p_0)$ to $(\SVI_{p_0})$, stemming
from properties of $\Fmap(p_0,\cdot):\X\rightrightarrows\Y$,
which will be employed in the sequel.

\begin{lemma}    \label{lem:Smapcloseconvex}
Let $(\SVI_p)$ be a parameterized family of set-valued inclusion
problems and let $p_0\in P$.
\begin{itemize}

\item[(i)] If $\Fmap(p_0,\cdot):\X\rightrightarrows\Y$ is l.s.c.
on $\X$, then $\Smap(p_0)$ is a closed set.

\item[(ii)] If $\Fmap(p_0,\cdot):\X\rightrightarrows\Y$ is
$C$-concave on $\X$, then $\Smap(p_0)$ is a convex set.

\end{itemize}
\end{lemma}

\begin{proof}
(i) It suffices to remember that $\Smap(p_0)=[\mef(p_0,\cdot)\le 0]$,
where $\mef$ is defined as in (\ref{eq:defmef}),
and to recall that, in the light of Remark \ref{rem:scontmef}(i), function
$\mef(p_0,\cdot)$ is l.s.c. on $\X$ whenever $\Fmap(p_0,\cdot)$ is so.

(ii) If $\Smap(p_0)=\varnothing$ the thesis trivially holds. Otherwise, take
arbitrary $x_1,\, x_2\in\Smap(p_0)$ and $t\in [0,1]$. By virtue of the
$C$-concavity of $\Fmap(p_0,\cdot)$, one finds
\begin{eqnarray*}
  \Fmap(p_0,tx_1+(1-t)x_2) &\subseteq & t\Fmap(p_0,x_1)+(1-t)\Fmap(p_0,x_2)+C \\
   &=& tC+(1-t)C+C=C,
\end{eqnarray*}
which means that $tx_1+(1-t)x_2\in\Smap(p_0)$, thereby proving the convexity
of $\Smap(p_0)$.
\end{proof}

In view of the formulation of the next result, it is convenient to
recall that a topological space $(P,\tau)$ is said to be paracompact
if every open cover of it admits an open refinement, which is locally finite, i.e.
for every point of $P$ there exists a neighbourhood that intersects
only finitely many members of such refinement. It is well known that
every compact topological space is also paracompact as well as every metrizable
topological space. On the other hand, there are paracompact topological
spaces which fail to be both compact and metrizable: consider, for
instance, $(\R,\tau)$, where $\tau$ stands for the lower limit topology
(a.k.a. right half-open interval topology,
see \cite[Part II, Counterexample 51]{SteSee78}).

The following theorem establishes the first of the main results of the paper.

\begin{theorem}[{\bf Global implicit function}]    \label{thm:gloimpfun}
With reference to a parameterized family of set-valued inclusion
problems $(\SVI_p)$, suppose that:

\begin{itemize}

\item[(i)] $(P,\tau)$ is a paracompact topological space and
$(\X,\|\cdot\|)$ is a Banach space;

\item[(ii)] for every $p\in P$ there exists $\hat x_p\in\X$ such that
$\exc{\Fmap(p,\hat x_p)}{C}<+\infty$;

\item[(iii)] $\Fmap(p,\cdot):\X\rightrightarrows\Y$ is l.s.c. on $\X$,
for every $p\in P$;

\item[(iv)] $\Fmap(p,\cdot):\X\rightrightarrows\Y$ is $C$-concave
on $\X$, for every $p\in P$;

\item[(v)] $\Fmap(\cdot,x):P\rightrightarrows\Y$ is Hausdorff $C$-u.s.c. on $P$,
for every $x\in\X$;

\item[(vi)] it holds $\alpha_{\Fmap,\Cmap}>1$.

\end{itemize}
Then, $\dom\Smap=P$ and there exists a function $\impf:P\longrightarrow\X$,
which is continuous on $P$ and such that
$$
  \impf(p)\in\Smap(p),\quad\forall p\in P.
$$
\end{theorem}

\begin{proof}
Under the above hypotheses it is clearly possible to invoke Proposition
\ref{pro:glosolv}, so that for any fixed $\alpha\in (1,\alpha_{\Fmap,\Cmap})$
and $z\in\X$, it is $\dom\Smap=P$ and the inequality
\begin{equation}    \label{in:erboSmapz}
  \dist{z}{\Smap(p)}\le {\exc{\Fmap(p,z)}{C}\over \alpha-1},
  \quad\forall p\in P.
\end{equation}
holds true.
Observe that, in the light of Lemma \ref{lem:Smapcloseconvex}(i),
$\Smap:P\rightrightarrows\X$ is closed valued as each
set-valued mapping $\Fmap(p,\cdot)$ is l.s.c. on $\X$.
Furthermore, by virtue of hypothesis (iv) and Lemma \ref{lem:Smapcloseconvex}(ii),
$\Smap$ turns out to take convex values.
Let us show that $\Smap$ is l.s.c. on $P$. To this aim, fix arbitrary
$p_0\in P$ and $A\subseteq\X$, with $A$ open, such that $\Smap(p_0)\cap A\ne
\varnothing$. Then take a point $x_0\in\Smap(p_0)\cap A$. As $A$ is open, there
exists $r_0>0$ such that $\ball{x_0}{r_0}\subseteq A$.
Since on account of hypothesis (v) $\Fmap(\cdot,x_0)$ is Hausdorff $C$-u.s.c. on $P$,
by remembering Remark \ref{rem:scontmef}(ii) one deduces that the scalar function
$p\mapsto\exc{\Fmap(\cdot,x_0)}{C}$ is u.s.c. in particular at $p_0$.
As a consequence, there exists a neighbourhood $U$ of $p_0$ such that
$$
  \exc{\Fmap(p,x_0)}{C}\le \exc{\Fmap(p_0,x_0)}{C}+{(\alpha-1)r_0\over 2}
  ={(\alpha-1)r_0\over 2},\quad\forall   p\in U.
$$
By taking into account inequality (\ref{in:erboSmapz}) with $z=x_0$,
from the last inequality one obtains
$$
  \dist{x_0}{\Smap(p)}\le {\exc{\Fmap(p,x_0)}{C}\over \alpha-1}\le
  {r_0\over 2},\quad\forall p\in U.
$$
This clearly implies
$$
  \ball{x_0}{r_0}\cap\Smap(p)\ne\varnothing,
  \quad\forall p\in U
$$
and hence, as it is $\ball{x_0}{r_0}\subseteq A$,
$$
  A\cap\Smap(p)\ne\varnothing,\quad\forall p\in U.
$$
Thus, by arbitrariness of $p_0\in P$, $\Smap$ is shown to be
l.s.c. on $P$. The above properties of $\Smap:P\rightrightarrows\X$,
along with the paracompactness of $P$ and the Banach space structure on $\X$
(both assumed in (i)), enable one to apply the Michael selection theorem (see
\cite{Mich56}). This ensures that existence of a global continuous
selection $\impf:P\longrightarrow\X$ of $\Smap$ and therefore completes
the proof.
\end{proof}

As a comment to Theorem \ref{thm:gloimpfun} it is worth noting
that all of its hypotheses, with the only exception of (vi), are merely
qualitative. In fact, on the parameter space $(P,\tau)$ no metric
structure is imposed. Consistently, no quantitative or enhanced form
of semicontinuity (such as Lipschitz lower/upper semicontinuity, calmness,
Aubin property) is invoked among the assumptions.
The only quantitative hypothesis concerns $\alpha_{\Fmap,\Cmap}$,
which is employed to ensure the lower semicontinuity of $\Smap$.
In contrast with that, all the existing stability results for parameterized
set-valued inclusions refer to quantitative properties of $\Smap$
(see the already mentioned \cite[Theorem 3.1,Theorem 3.3, Theorem 3.5]{Uder21}).

\begin{example}
Let $(P,\tau)$ be a paracompact topological space, let $\X=\R^n$,
$\Y=\R^m$ and $C=\R^m_+$, with $n\ge m\ge 2$. Let us consider a
continuous function $M:P\longrightarrow\Mat_{m\times n}(\R)$
such that
\begin{equation}    \label{eq:incrbosigma}
  \sigma_M=\inf_{p\in P}\inc{M(p)}{\R^m_+}{\nullv}>1.
\end{equation}
Let $h=(h_1,\dots,h_m):\R^n\longrightarrow\R^m$ be any function
Lipschitz continuous on $\R^n$, with constant $\ell_h$, and
having each component $h_i:\R^n\longrightarrow\R$ concave on $\R^n$,
$i=1,\dots,m$.
Let $\Hmap_L:\R^n\rightrightarrows\R^m$ be a fan which is compactly
generated by $L\subset\Mat_{m\times n}(\R)$ (see Example \ref{ex:comgenfan}),
with $\ell_{\Hmap_L}$ given by (\ref{eq:fannorm}). Suppose that
$$
  \ell_h+\ell_\Hmap<1-\frac{1}{\sigma_M}.
$$
Observe that, since $h$ and $\Hmap_L$ are Lipschitz continuous on $\R^n$,
then also their sum $h+\Hmap_L$ is Lipschitz continuous,
with constant $\ell_h+\ell_{\Hmap_L}$.
Given the above data, consider the  family of problems $(\SVI_p)$
\begin{equation}\label{eqin:SVIpproex}
  \Fmap(p,x)=M(p)x+h(x)+\Hmap_L(x)\subseteq\R^m_+.
\end{equation}
Let us check that such a class of problems fulfils all the hypotheses
of Theorem \ref{thm:gloimpfun}.
Fixed any $p\in P$, as $h+\Hmap_L$ is Lipschitz continuous on $\R^n$,
then $\Fmap(p,\cdot)=M(p)+h+\Hmap_L$ is Lipschitz continuous as well.
As a consequence, $\Fmap(p,\cdot)$ is a fortiori l.s.c. on $\R^n$.
Since $h:\R^n\longrightarrow\R^m$ is $\R^m_+$-concave, so is $M(p)+h$.
Therefore, as $\Hmap_L$ is $\R^m_+$-concave, the same holds true for
the sum $\Fmap(p,\cdot)$.
The fact that $\Hmap_L$ takes bounded values, $L$ being bounded in $\Mat_{m\times}
(\R)$, entails that so does also $\Fmap$, what ensures the satisfaction
of  hypothesis (ii) in Theorem \ref{thm:gloimpfun}.
Since $M$ is continuous on $P$, its translation $\Fmap(\cdot,x)$
by the constant (with respect to $p$) set $h(x)+\Hmap_L(x)$
is Hausdorff $\R^m_+$-u.s.c. on $P$, for every $x\in\R^n$.
Finally, by recalling the condition in (\ref{eq:incrbosigma}), one has that
for every $p\in P$ the linear mapping represented by $M(p)$ is
$\R^m_+$-increasing on $\R^n$, with $\inc{M(p)}{R^m_+}{x}
\ge\sigma_M$ for every $x\in\R^n$. As it is $\ell_h+\ell_{\Hmap_L}<1-\frac{1}{\sigma_M}$,
Proposition \ref{pro:Cincrprosta} enables one to assert that the
additive perturbation $\Fmap(p,\cdot)$ is still $\R^m_+$-increasing
on $\R^n$, with
\begin{eqnarray*}
  \inc{\Fmap(p,\cdot)}{\R^m_+}{x} &\ge & (1-\ell_h-\ell_{\Hmap_L})\inc{M(p)}{\R^m_+}{x}  \\
   &\ge& (1-\ell_h-\ell_{\Hmap_L})\sigma_M>\frac{1}{\sigma_M}\sigma_M=1,
   \quad\forall x\in\R^n,
\end{eqnarray*}
and therefore
$$
  \alpha_{\Fmap,\R^m_+}\ge \inf_{(p,x)\in P\times\R^n}\inc{M(p))}{\R^m_+}{x}\ge
  (1-\ell_h-\ell_{\Hmap_L})\sigma_M>1.
$$
This completes the check of the validity of all the hypotheses in
Theorem \ref{thm:gloimpfun}. It is possible to conclude that
(\ref{eqin:SVIpproex}) singles out a class of $(\SVI_p)$ to which
Theorem \ref{thm:gloimpfun} can be applied.

In order to check up the validity of Theorem \ref{thm:gloimpfun}
by direct computation of solutions to $(\SVI_p)$ in some concrete cases,
let us consider a specific instance of (\ref{eqin:SVIpproex}).
Let us take $P=\R$ equipped with its usual structure, and $m=n=2$.
Consider the mapping $M:\R\longrightarrow\Mat_{2\times 2}(\R)$,
defined by
$$
  M(p)=3O_p=3\left(\begin{array}{cc}
                  \cos p & -\sin p \\
                  \sin p & \cos p
               \end{array}\right).
$$
Its continuous dependence on $p$ is a straightforward consequence
of the continuity of both the functions $p\mapsto\sin p$ and
$p\mapsto\cos p$.
Notice that each matrix $3O_p$ represents a rotation of $\R^2$
rescaled by a factor $3$, as discussed in Example \ref{ex:rescarot}.
Therefore, on the base of the estimates obtained in that example, one has
$$
  \inc{M(p)}{\R^2_+}{x}=\frac{3}{\sqrt{2}}+1,\quad\forall
  (p,x)\in\R\times\R^2.
$$
Thus, in the current case, according to (\ref{in:incFmapCX}) it holds
$$
  \alpha_{\Fmap,\R^2_+}=\frac{3}{\sqrt{2}}+1>1.
$$
Choose $h:\R^2\longrightarrow\R^2$, defined by
$$
  h(x)=\left(\begin{array}{c}
               -1-\displaystyle\frac{|x_1|}{4} \\
               -1-\displaystyle\frac{|x_2|}{4}
             \end{array}\right),  \qquad x=(x_1,x_2),
$$
and the fan $\Hmap_L:\R^2\rightrightarrows\R^2$, generated by the following
convex compact subset of $\Mat_{2\times 2}(\R)$
$$
  L=\left\{\left(\begin{array}{cc}
                  \lambda & 0 \\
                  0 & \lambda
               \end{array}\right):\ \lambda\in
               \left[-\frac{1}{4},\frac{1}{4}\right]\right\}.
$$
The above data single out the following instance of the problem
class $(\SVI_p)$: find $x\in\R^2$ such that
\begin{equation}    \label{in:speinstSVIp}
    3\left(\begin{array}{cc}
                  \cos p & -\sin p \\
                  \sin p & \cos p
     \end{array}\right)\binom{x_1}{x_2}+
               \left(\begin{array}{c}
               -1-\displaystyle\frac{|x_1|}{4} \\
               -1-\displaystyle\frac{|x_2|}{4}
     \end{array}\right)+\left[-\frac{1}{4},\frac{1}{4}\right]
     \binom{x_1}{x_2}\subseteq\R^2_+,
\end{equation}
where $[-1/4,1/4]x=\{\lambda x:\ \lambda\in [-1/4,1/4]\}$.
It is plain to see that, fixed $p\in\R$, the set-valued inclusion
problem (\ref{in:speinstSVIp})
is solved by any $x\in\R^2$, which satisfies the following inequality
system:
$$
   \left\{\begin{array}{cc}
       3(\cos p)x_1-3(\sin p)x_2\ge 1+\displaystyle\frac{|x_1|}{4}+\displaystyle\frac{|x_1|}{4}  \\
       \\
       3(\sin p)x_1+3(\cos p)x_2\ge 1+\displaystyle\frac{|x_2|}{4}+\displaystyle\frac{|x_2|}{4}.
   \end{array}\right.
$$
Thus, by defining $\impf:\R\longrightarrow\R^2$ as being
$$
  \impf(p)=O_{\frac{\pi}{4}-p}\binom{1}{0}=
  \binom{\cos\left(\frac{\pi}{4}-p\right)}{\sin\left(\frac{\pi}{4}-p\right)},
$$
one obtains
$$
  3O_p\impf(p)=\left[3O_p\circ O_{\frac{\pi}{4}-p}\right]\binom{1}{0}=
  3O_{\frac{\pi}{4}}\binom{1}{0}=\binom{3/\sqrt{2}}{3/\sqrt{2}}.
$$
On the other hand, as it is $\impf(p)=(\impf_1(p),\impf_2(p))\in\Usfer$
for every $p\in\R$, one finds
$$
   \left\{\begin{array}{cc}
       \displaystyle\frac{3}{\sqrt{2}}>\displaystyle\frac{3}{2}\ge
       1+\displaystyle\frac{|\impf_1(p)|}{4}+\displaystyle\frac{|\impf_1(p)|}{4}  \\
       \\
       \displaystyle\frac{3}{\sqrt{2}}>\displaystyle\frac{3}{2}\ge
       1+\displaystyle\frac{|\impf_2(p)|}{4}+\displaystyle\frac{|\impf_2(p)|}{4},
   \end{array}\right. \qquad\forall p\in\R,
$$
which shows that $x=\impf(p)$ actually solves the above inequality
system for every $p\in\R$.
Consistently with the thesis of Theorem \ref{thm:gloimpfun}, one can
see that an (explicit) solution $x=\impf(p)$ does exist and is continuous on $\R$.
\end{example}

\vskip1cm


\subsection{The constrained case}

In the present subsection, the perturbation analysis so far
conducted for $(\SVI_p)$ problems is extended to the more involved
case of $(\CSVI_p)$. The presence of a constraint mapping $\Rmap$
depending on the perturbation parameter makes such an extension
not trivial, when dealing with continuity of the implicit function.

Throughout the current subsection, along with the standing assumptions
$(\mathfrak{A}_0)$--$(\mathfrak{A}_2)$,
the additional assumption will be maintained:

\begin{itemize}
  \item [$(\mathfrak{A}_3)$] $(X,d)$ is metrically convex.
\end{itemize}

Recall that a metric space $(X,d)$ is said to be metrically convex if for any
pair of distinct points $a,\, b\in X$ there exists $c\in X$ such that
$d(a,b)=d(a,c)+d(c,b)$. Clearly, every convex subset of a normed vector
space is metrically convex, provided that it is equipped with the
norm distance.

A property of the distance function to a
given subset of a metrically convex space, which will be employed
in the sequel, is pointed out in the next lemma.
In view of its statement, let us mention
that a subset $S\subseteq X$ of a metric space is called proximinal
if for every $x\in X$ there exists at least one $s_x\in S$ such that
$d(x,s_x)=\dist{x}{S}$. As a straightforward consequence of such a
definition, one has that every proximinal set is nonempty and closed.
On the other hand, every nonempty, closed and convex subset of a
reflexive Banach space is known to be proximinal.

\begin{lemma}   \label{lem:metsegprox}
Let $(X,d)$ be a metrically convex, complete metric space, let
$S\subseteq X$ a proximinal set, and let $x\in X\backslash S$.
Then for every $r\in (0,\dist{x}{S})$ there exists $u\in\ball{x}{r}\backslash
\{x\}$, such that
$$
  \dist{u}{S}=\dist{x}{S}-d(u,x).
$$
\end{lemma}

\begin{proof}
In the light of the Menger theorem (see \cite[Theorem 1.97]{Peno13}),
the metric space $(X,d)$ turns out to be also a metric segment space.
This means that for any pair of distinct points $a,\, b\in X$ there exists a geodesic
joining them, i.e. an isometric mapping $g:[0,d(a,b)]\longrightarrow
g([0,d(a,b)])$, such that $g(0)=a$ and $g(d(a,b))=b$.
Now, since $S$ is proximinal and $x\not\in S$, there exists $s_x\in S$
such that $d_S=\dist{x}{S}=d(x,s_x)>0$. Thus, corresponding to the pair
$x,\, s_x$, there exists a geodesic $g:[0,d_S]\longrightarrow g([0,d_S])$,
such that $g(0)=x$ and $g(d_S)=s_x$.
Take an arbitrary  $r\in (0,d_S)$, choose $t\in (0,r)$ and set $u=g(t)$.
Observe that, as it is $d(u,x)=d(g(t),g(0))=t\in (0,r)$, then one
has $u\in\ball{x}{r}\backslash\{x\}$. Moreover, since it holds
$$
  d(u,s_x)=d(g(t),g(d_S))=d_S-t,
$$
it results in
$$
  \dist{u}{S}\le d(u,s_x)=\dist{x}{S}-d(u,x).
$$
The last inequality, along with the following one
$$
  \dist{x}{S}\le d(x,u)+\dist{u}{S},
$$
gives the equality in the thesis, thereby completing the proof.
\end{proof}

The next proposition provides sufficient conditions for the existence
of solutions to constrained set-valued inclusion problems in the presence
of global parameter perturbations, which are expressed in terms of the
following constant:

\begin{equation}    \label{in:incFmapCXR}
 \widetilde{\alpha}_{\Fmap,\Cmap,\Rmap}=\inf\bigg\{\inc{\Fmap(p,\cdot)}{\Cmap(p)}{x}:\
 (p,x)\in P\times X,\ \Fmap(p,x)\not\subseteq\Cmap(p),\ x\in\Rmap(p)\bigg\}.
\end{equation}

\begin{proposition}[{\bf Constrained global solvability}]   \label{pro:gloconsolv}
With reference to a parameterized family of set-valued inclusion
problems $(\CSVI_p)$, suppose that:

\begin{itemize}

\item[(i)] for every $p\in P$ there exists $\hat x_p\in X$ such that
$\exc{\Fmap(p,\hat x_p)}{\Cmap(p)}<+\infty$;

\item[(ii)] it is $\widetilde{\alpha}_{\Fmap,\Cmap,\Rmap}>1$;

\item[(iii)] $\Fmap(p,\cdot):X\rightrightarrows\Y$ is l.s.c. on $X$
for every $p\in P$ and is Lipschitz continuous
on $X\backslash\Rmap(p)$, with a constant $0<\ell<\widetilde{\alpha}_{\Fmap,\Cmap,\Rmap}-1$,
uniform on $P$;

\item[(iv)] $\Rmap(p)$ is proximinal, for every $p\in P$.

\end{itemize}
Then, $\dom\CSmap=P$ and for any $\alpha$, with
$\displaystyle\frac{\widetilde{\alpha}_{\Fmap,\Cmap,\Rmap}-\ell+1}{2}
<\alpha<\widetilde{\alpha}_{\Fmap,\Cmap,\Rmap}-\ell$,
and any $x_0\in X$ it holds
\begin{equation}   \label{in:gloconsolvthesis}
  \dist{x_0}{\CSmap(p)}\le
  {\exc{\Fmap(p,x_0)}{\Cmap(p)}+(\widetilde{\alpha}_{\Fmap,\Cmap,\Rmap}-\alpha)
  \dist{x_0}{\Rmap(p)}\over
  \widetilde{\alpha}_{\Fmap,\Cmap,\Rmap}-\alpha-\ell},\quad\forall p\in P.
\end{equation}
\end{proposition}

\begin{proof}
First of all observe that the inequality
$$
  \displaystyle\frac{\widetilde{\alpha}_{\Fmap,\Cmap,\Rmap}-\ell+1}{2}
  <\widetilde{\alpha}_{\Fmap,\Cmap,\Rmap}-\ell,
$$
actually holds true,
because it has been supposed $\ell<\widetilde{\alpha}_{\Fmap,\Cmap,\Rmap}-1$.
Besides, as $\alpha<\widetilde{\alpha}_{\Fmap,\Cmap,\Rmap}-\ell$, it is also
$\widetilde{\alpha}_{\Fmap,\Cmap,\Rmap}-\alpha>0$ and $\widetilde{\alpha}_{\Fmap,\Cmap,\Rmap}
-\alpha-\ell>0$, so that the error bound in (\ref{in:gloconsolvthesis})
does make sense.

Fix $x_0\in X$ and the value of $\alpha$ as prescribed in the thesis. Take an
arbitrary $p\in P$. If $x_0\in\CSmap(p)$ all the assertions in the thesis
come true. So, assume that $x_0\not\in\CSmap(p)$ and consider the merit
function $\cmef_p:X\longrightarrow [0,+\infty]$, defined by
$$
  \cmef_p(x)=\mef_p(x)+(\widetilde{\alpha}_{\Fmap,\Cmap,\Rmap}-\alpha)
  \dist{x}{\Rmap(p)},
$$
where $\mef_p$ is as in (\ref{eq:defmefp}). Since $\Rmap(p)$ is nonempty
as a proximinal set, hypothesis (i) guarantees that $\dom\cmef_p\ne\varnothing$.
Since according to Remark \ref{rem:scontmef}(i) $\mef_p$ is l.s.c and function
$x\mapsto(\widetilde{\alpha}_{\Fmap,\Cmap,\Rmap}-\alpha)\dist{x}{\Rmap(p)}$
is Lipschitz continuous on $X$, then their sum $\cmef_p$ is l.s.c. on $X$.
In order to reproduce the argument employed in Proposition \ref{pro:glosolv},
it remains to show that, under the current hypotheses, $\cmef_p$ fulfils
the Caristi-type condition $(\mathfrak{C})$.
To this aim, let us consider an arbitrary $x\in X$, such that
$\cmef_p(x)>0$. If this inequality is true, two cases only are possible:

\vskip.25cm

\noindent{\sc case I:} $\mef_p(x)>0$ and $x\in\Rmap(p)$.

In such an event, it is $\Fmap(p,x)\not\subseteq\Cmap(p)$, so
one can proceed by adapting to the current circumstance
the proof of Proposition \ref{pro:glosolv}. Accordingly, let $\delta_*>0$
be such that $\mef_p(z)>0$ for every  $z\in\ball{x}{\delta_*}$. Recall that
$\Fmap(p,\cdot)$ is $\Cmap(p)$-increasing at $x$ and notice that, as it was
chosen
$$
  \alpha>\frac{\widetilde{\alpha}_{\Fmap,\Cmap,\Rmap}-\ell+1}{2},
$$
one has
$$
  2(\widetilde{\alpha}_{\Fmap,\Cmap,\Rmap}-\alpha)-\ell+1
  <\widetilde{\alpha}_{\Fmap,\Cmap,\Rmap}.
$$
Moreover it holds
$$
   2(\widetilde{\alpha}_{\Fmap,\Cmap,\Rmap}-\alpha)-\ell+1>1,
$$
because $\widetilde{\alpha}_{\Fmap,\Cmap,\Rmap}-\alpha>\ell>\frac{\ell}{2}$.
Thus, according to Definition \ref{def:Cincrpropdef}, there exists $\bar\delta_\alpha>0$
such that for every $r\in (0,\bar\delta_\alpha]$ there is $u\in\ball{x}{r}$
such that
$$
  \ball{\Fmap(p,u)}{[2(\widetilde{\alpha}_{\Fmap,\Cmap,\Rmap}-\alpha)-\ell+1]r}
  \subseteq \ball{\Fmap(p,x)+\Cmap(p)}{r}.
$$
Then, if choosing $r\in (0,\min\{\delta_*,\bar\delta_\alpha\})$ and
arguing as in the proof of Proposition \ref{pro:glosolv}, one gets the
existence of $u\in\ball{x}{r}\backslash\{x\}$ such that
\begin{eqnarray*}
  \mef_p(u) &=& \exc{\ball{\Fmap(p,u)}{[2(\widetilde{\alpha}_{\Fmap,\Cmap,\Rmap}-\alpha)-\ell+1]r}}{\Cmap(p)} \\
  & & -[2(\widetilde{\alpha}_{\Fmap,\Cmap,\Rmap}-\alpha)-\ell+1]r \\
   &\le& \mef_p(x)-[2(\widetilde{\alpha}_{\Fmap,\Cmap,\Rmap}-\alpha)-\ell]r.
\end{eqnarray*}
As it is $d(u,x)\le r$, from the above inequality it follows
\begin{equation}    \label{in:mefpCaristiapa}
  \mef_p(u)\le\mef_p(x)-[2(\widetilde{\alpha}_{\Fmap,\Cmap,\Rmap}-\alpha)-\ell]d(u,x).
\end{equation}
By taking into account that the function
$x\mapsto(\widetilde{\alpha}_{\Fmap,\Cmap,\Rmap}-\alpha)\dist{x}{\Rmap(p)}$
is Lipschitz continuous on $X$ with constant $\widetilde{\alpha}_{\Fmap,\Cmap,\Rmap}-\alpha$,
the inequality in (\ref{in:mefpCaristiapa}) implies
\begin{eqnarray*}
  \cmef_p(u)+(\widetilde{\alpha}_{\Fmap,\Cmap,\Rmap}-\alpha-\ell)d(u,x) &=&
  \mef_p(u)+(\widetilde{\alpha}_{\Fmap,\Cmap,\Rmap}-\alpha)\dist{u}{\Rmap(p)}  \\
   & & +(\widetilde{\alpha}_{\Fmap,\Cmap,\Rmap}-\alpha-\ell)d(u,x) \\
   &\le& \mef_p(u)+(\widetilde{\alpha}_{\Fmap,\Cmap,\Rmap}-\alpha)\dist{x}{\Rmap(p)} \\
   & & +(\widetilde{\alpha}_{\Fmap,\Cmap,\Rmap}-\alpha)d(u,x)  \\
   & & +(\widetilde{\alpha}_{\Fmap,\Cmap,\Rmap}-\alpha-\ell)d(u,x)  \\
   &\le& \mef_p(x)-[2(\widetilde{\alpha}_{\Fmap,\Cmap,\Rmap}-\alpha)-\ell]d(u,x) \\
   & & +(\widetilde{\alpha}_{\Fmap,\Cmap,\Rmap}-\alpha)\dist{x}{\Rmap(p)}  \\
   & & +(\widetilde{\alpha}_{\Fmap,\Cmap,\Rmap}-\alpha)d(u,x)  \\
   & & +(\widetilde{\alpha}_{\Fmap,\Cmap,\Rmap}-\alpha-\ell)d(u,x)  \\
   &=& \mef_p(x)+(\widetilde{\alpha}_{\Fmap,\Cmap,\Rmap}-\alpha)\dist{x}{\Rmap(p)} \\
   &=& \cmef_p(x).
\end{eqnarray*}
This inequality shows that $\cmef_p$ fulfils the Caristi-type
condition $(\mathfrak{C})$, with $\hat x=u$
and $\kappa=\widetilde{\alpha}_{\Fmap,\Cmap,\Rmap}-\alpha-\ell$. Consequently,
what recalled in Remark \ref{rem:Carcond} ensures the existence of $x_\kappa\in X$
such that
\begin{equation}   \label{eq:xksolCSVIp}
  \mef_p(x_\kappa)=0, \qquad \dist{x_\kappa}{\Rmap(p)}=0,
\end{equation}
and
\begin{equation}    \label{in:xksolCSVIp}
  d(x_\kappa,x_0)\le
  \frac{\cmef_p(x_0)}{\widetilde{\alpha}_{\Fmap,\Cmap,\Rmap}-\alpha-\ell}.
\end{equation}
As $\Cmap(p)$ and $\Rmap(p)$ are closed sets,
from (\ref{eq:xksolCSVIp}) one deduces that $x_\kappa\in\CSmap(p)$ and,
on account of (\ref{in:xksolCSVIp}), one achieves the inequality in the
thesis.

\vskip.25cm

\noindent{\sc case II:} $x\in X\backslash\Rmap(p)$ and hence $\dist{x}{\Rmap(p)}>0$.

Since by hypothesis (iii) $\Fmap:P\times X\rightrightarrows\Y$ is Lipschitz
continuous with respect to $x$, uniformly on $P$, with constant $\ell$,
Remark \ref{rem:scontmef}(iii) ensures that function $\mef_p$ turns out to be
Lipschitz continuous with constant $\ell$. Thus, one has
\begin{equation}    \label{in:mefcalmabx}
  \mef_p(z)-\mef_p(x)\le\ell d(z,x),\quad\forall
  z\in  X\backslash\Rmap(p).
\end{equation}
Since $\Rmap(p)$ is supposed to be proximinal (by hypothesis (iv)), there
exists $z_x\in\Rmap(p)$ such that $d(z_x,x)=\dist{x}{\Rmap(p)}$.
By recalling that $X$ is a metrically convex space (remember assumption
$(\mathfrak{A}_3)$) so Lemma \ref{lem:metsegprox} applies, one
can state the existence of $u\in X\backslash\{x\}$ such that
$$
   \dist{u}{\Rmap(p)}=\dist{x}{\Rmap(p)}-d(u,x).
$$
Observe that, as $u$ can be chosen sufficiently close to $x$, it is
$u\in X\backslash\Rmap(p)$.
As a consequence, on account of the inequality in (\ref{in:mefcalmabx}),
one obtains
\begin{eqnarray*}
  \cmef_p(u) &=& \mef_p(u)+(\widetilde{\alpha}_{\Fmap,\Cmap,\Rmap}-\alpha)\dist{u}{\Rmap(p)} \\
   &=& \mef_p(u)+(\widetilde{\alpha}_{\Fmap,\Cmap,\Rmap}-\alpha)\dist{x}{\Rmap(p)}
     -(\widetilde{\alpha}_{\Fmap,\Cmap,\Rmap}-\alpha)d(u,x) \\
   &\le& \mef_p(x)+\ell d(u,x)+(\widetilde{\alpha}_{\Fmap,\Cmap,\Rmap}-\alpha)\dist{x}{\Rmap(p)}
      -(\widetilde{\alpha}_{\Fmap,\Cmap,\Rmap}-\alpha)d(u,x)   \\
   &=& \cmef_p(x)-[\widetilde{\alpha}_{\Fmap,\Cmap,\Rmap}-\alpha-\ell]d(u,x),
\end{eqnarray*}
whence it follows
$$
 \cmef_p(u)+[\widetilde{\alpha}_{\Fmap,\Cmap,\Rmap}-\alpha-\ell]d(u,x)\le\cmef_p(x).
$$
With the last inequality, $\cmef_p$ is shown to fulfil the Caristi-type condition $(\mathfrak{C})$
also in the current case, so the same conclusions as in previous case
can be drawn. This completes the proof.
\end{proof}

\begin{remark}
It is worth observing that, in the particular case of $p\in\Rmap^{-1}(x_0)$,
the inequality in (\ref{in:gloconsolvthesis}) becomes
$$
   \dist{x_0}{\CSmap(p)}\le
  {\exc{\Fmap(p,x_0)}{\Cmap(p)}\over\widetilde{\alpha}_{\Fmap,\Cmap,\Rmap}-\alpha-\ell}.
$$
If taking into account that $\alpha>\frac{1}{2}(\widetilde{\alpha}_{\Fmap,\Cmap,\Rmap}-\ell+1)$
and hence $\widetilde{\alpha}_{\Fmap,\Cmap,\Rmap}-\alpha-\ell<\alpha-1$, one sees
that the resulting error bound estimate is worse than the one
provided in (\ref{in:glosolvthesis}), that is in the unconstrained
case. Thus, the major complication of $(\CSVI_p)$ in comparison
with $(\SVI_p)$, caused by the presence of constraints, results in
a looser error bound in case $p\in\Rmap^{-1}(x_0)$.
On the other hand, whenever $\Rmap$ happens to take the constant
value $X$ (as in the unconstrained case), then case II in the proof
of Proposition \ref{pro:gloconsolv} can not take place. In such an
event, hypothesis (iii) of that proposition becomes redundant and
can be considerably relaxed. In fact, the Lipschitz continuity of
$\Fmap$ and the relation between $\ell$ and $\widetilde{\alpha}_{\Fmap,\Cmap,\Rmap}$
are not actually employed in case I, whereas a lower semicontinuity
assumption on $\Fmap(p,\cdot)$ seems to be adequate.
Therefore, in this special
circumstance an improvement of the error bound estimate is expected.
To summarize, the price to be paid for handling problems in a
major generality consists in a minor accuracy of the error bound
estimate.
\end{remark}

The conditions for global solvability achieved above lead to the
following implicit function theorem for $(\CSVI_p)$. Like in the
unconstrained case, it will be established in the particular
case in which $\Cmap$ takes the constant value $C\subset\Y$,
while the metric space $X$ is specialized to be a (complete,
according to $(\mathfrak{A}_0)$) normed space $\X$,
because of the employment of convexity assumptions.

\begin{theorem}[{\bf Global implicit function}]    \label{thm:congloimpfun}
With reference to a parameterized family of set-valued inclusion
problems $(\CSVI_p)$, suppose that:

\begin{itemize}

\item[(i)] $(P,\tau)$ is a paracompact topological space and
$(\X,\|\cdot\|)$ is a Banach space;

\item[(ii)] for every $p\in P$ there exists $\hat x_p\in\X$ such that
$\exc{\Fmap(p,\hat x_p)}{C}<+\infty$;

\item[(iii)] it is $\widetilde{\alpha}_{\Fmap,\Cmap,\Rmap}>1$;

\item[(iv)] $\Fmap(p,\cdot):\X\rightrightarrows\Y$ is l.s.c. on $\X$
for every $p\in P$, and Lipschitz continuous on $\X\backslash\Rmap(p)$,
with a constant $0<\ell<\widetilde{\alpha}_{\Fmap,\Cmap,\Rmap}-1$
uniform on $P$;

\item[(v)] $\Fmap(p,\cdot):\X\rightrightarrows\Y$ is $C$-concave
on $\X$, for every $p\in P$;

\item[(vi)] $\Fmap(\cdot,x):P\rightrightarrows\Y$ is Hausdorff $C$-u.s.c. on $P$,
for every $x\in\X$;

\item[(vii)] $\Rmap:P\rightrightarrows\X$ is l.s.c. on $P$ and
takes proximinal convex values.

\end{itemize}
Then, $\dom\CSmap=P$ and there exists a function $\cimpf:P\longrightarrow\X$,
which is continuous on $P$ and such that
$$
  \cimpf(p)\in\CSmap(p),\quad\forall p\in P.
$$
\end{theorem}

\begin{proof}
Since it is possible to invoke Proposition \ref{pro:gloconsolv},
all the required hypotheses being satisfied, one has that $\dom\CSmap
=P$ and, fixed $\alpha$ as in the thesis of the aforementioned
proposition, for any $z\in\X$ it holds
\begin{equation}    \label{in:erboconsolv}
  \dist{z}{\CSmap(p)}\le
  {\exc{\Fmap(p,z)}{C}+(\widetilde{\alpha}_{\Fmap,\Cmap,\Rmap}-\alpha)
  \dist{z}{\Rmap(p)}\over
  \widetilde{\alpha}_{\Fmap,\Cmap,\Rmap}-\alpha-\ell},\quad\forall p\in P.
\end{equation}
In order to apply the Michael selection theorem, one has to check
that $\CSmap$ takes closed and convex values and that $\CSmap:P\rightrightarrows\X$
is l.s.c. on $P$. So, fix an arbitrary $p\in P$ and consider
the function $\cmef:P\times\X\longrightarrow[0,+\infty]$ as defined
by
$$
  \cmef(p,x)={\exc{\Fmap(p,x)}{C}+(\widetilde{\alpha}_{\Fmap,\Cmap,\Rmap}-\alpha)
  \dist{x}{\Rmap(p)}\over
  \widetilde{\alpha}_{\Fmap,\Cmap,\Rmap}-\alpha-\ell}.
$$
Since $\CSmap(p)=[\cmef(p,\cdot)\le 0]$ and,
under the current hypotheses, $\cmef(p,\cdot)$ is Lipschitz continuous on $\X$
(remember the proof of Proposition \ref{pro:gloconsolv}), then set
$\CSmap(p)$ is closed. In a similar manner, by observing that the functions
$x\mapsto\exc{\Fmap(p,x)}{C}$ and $x\mapsto\dist{x}{\Rmap(p)}$ are both
convex on $\X$ by the $C$-concavity of $\Fmap(p,\cdot)$ and the
convexity of $\Rmap(p)$, respectively, one sees that also their convex
combination $\cmef(p,\cdot)$ is convex, with the consequence that
its sublevel set $\CSmap(p)$ is convex.
The lower semicontinuity of $\CSmap$ can be shown by adapting the
argument exposed in the proof of Theorem \ref{thm:gloimpfun}.
To see how in detail, fix arbitrary $p_0\in P$ and $A\subseteq\X$,
with $A$ open and such that $\CSmap(p_0)\cap A\ne\varnothing$.
If $x_0\in\CSmap(p_0)\cap A$, there exists $r_0>0$ such that
$\ball{x_0}{r_0}\subseteq A$. Since by virtue of hypothesis (vi)
the set-valued mapping $\Fmap(\cdot,x_0)$ is Hausdorff $C$-u.s.c.
on $P$ and, by virtue of hypothesis (vii), $\Rmap$ is l.s.c. on $P$,
the function $p\mapsto\cmef(p,x_0)=\exc{\Fmap(p,x_0)}{C}+
(\widetilde{\alpha}_{\Fmap,C,\Rmap}-\alpha)\dist{x_0}{\Rmap(p)}$
turns out to be u.s.c., in particular at $p_0$ (recall Remark
\ref{rem:scontmef}(ii) and Remark \ref{rem:uscdistp}).
This means that there exists a neighbourhood $U$ of $p_0$ such that
$$
  \cmef(p,x_0) \le \cmef(p_0,x_0)+
  \frac{(\widetilde{\alpha}_{\Fmap,\Cmap,\Rmap}-\alpha-\ell)r_0}{2}
  =\frac{(\widetilde{\alpha}_{\Fmap,\Cmap,\Rmap}-\alpha-\ell)r_0}{2},
  \quad\forall p\in U.
$$
Thus, by exploting the error bound in (\ref{in:erboconsolv})
with $z=x_0$, one finds
$$
  \dist{x_0}{\CSmap(p)}\le{\cmef(p,x_0)\over \widetilde{\alpha}_{\Fmap,\Cmap,\Rmap}-\alpha-\ell}
  \le\frac{r_0}{2},\quad\forall p\in U.
$$
This inequality entails $\ball{x_0}{r_0}\cap\CSmap(p)\ne\varnothing$ and
hence $A\cap\CSmap(p)\ne\varnothing$ for every $p\in U$.
Thus, by applying the Michael selection theorem one achieves
all the assertions in the thesis.
\end{proof}

\vskip1cm


\section{Applications to the stability analysis of ideal efficiency
in parametric vector optimization} \label{Sect:4}

With the aim of illustrating the application range of the theory exposed
in Section \ref{Sect:3}, in the present section a qualitative analysis
of the solution stability is conducted for finite-dimensional
vector optimization problems that can be put in the following
form
$$
  \Cmap(p)\hbox{-}\min f(p,x)\ \hbox{ subject to }\ x\in\Rmap(p). \leqno (\VOP_p)
$$
The data defining the above parameterized class of vector optimization problems are
a constraining set-valued mapping $\Rmap:P\rightrightarrows\R^n$, which
describes the feasible region of $(\VOP_p)$ under parameter perturbation; a set-valued mapping
$\Cmap:P\rightrightarrows\R^m$ expressing a partial order over $\R^m$
in the standard way (i.e. $y_1\le_{\Cmap(p)}y_2$ iff $y_2-y_1\in\Cmap(p)$),
which varies with the parameter; a mapping $f:P\times\R^n\longrightarrow\R^m$
representing the vector objective function. Problems of this kind are
the subject of comprehensive theoretical studies in multicriteria optimization
and multiobjective programming (see \cite{Jahn11,Luc89,SaNaTa85}).

Throughout the current section, to the standing assumptions $(\mathfrak{A}_0)-(\mathfrak{A}_2)$
the following ones are added:
\begin{itemize}

\item[$(\mathfrak{A}_4)$] $\Cmap:P\rightrightarrows\Y$ takes pointed values;

\item[$(\mathfrak{A}_5)$] $\Rmap:P\rightrightarrows\Y$ takes nonempty and closed values.

\end{itemize}
Notice that $(\mathfrak{A}_0)$ and $(\mathfrak{A}_3)$ automatically hold, as $X=\R^n$. Assumption
$(\mathfrak{A}_5)$ allows to guarantee, in the finite-dimensional setting,
the proximinality of the values of $\Rmap$ once and for all.

Among various solution notions one can consider in connection with
vector optimization problems, here the focus will be on ideal efficient solutions.
Fixed $\bar p\in P$, $\bar x\in\Rmap(\bar p)$ is said to be a (global) ideal efficient
solution to $(\VOP_ {\bar p})$ if
$$
   f(\bar p,\bar x)\le_{\Cmap(\bar p)}f(\bar p,x),\quad\forall
   x\in\Rmap(\bar p)
$$
or, equivalently, it holds
$$
  f(\bar p,\Rmap(\bar p))-f(\bar p,\bar x)\subseteq\Cmap(\bar p).
$$
While the perturbation analysis of vector optimization problems
has mainly concentrated on efficient and weak efficient solutions
(see, among the other, \cite{Bedn95,CaLoMoPa20,CraLuu00,HuMoYa08,Luc89,SaNaTa85,Tani90}),
much less is known about the stability properties of ideal efficiency.
Solvability results in the case of ideal efficiency are presented in \cite{Flor02}.
Some recent results about the quantitative stability (in terms
of Lipschitz lower semicontinuity and calmness) of ideal
efficient solutions can be found in \cite{Uder23}. The investigations
here exposed can be regarded as a complement of them.
As the focus is on the ideal efficiency, the solution set-valued mapping
$\IEmap:P\rightrightarrows\R^n$ associated with $(\VOP_ {\bar p})$
becomes
$$
  \IEmap(p)=\{x\in\Rmap(p):\ f(p,\Rmap(p))-f(p,x)
  \subseteq\Cmap(p)\}.
$$
Along with the solution mapping, another fundamental element of the perturbation
analysis is the value mapping associated with $(\VOP_p)$, henceforth denoted by
$\val$, i.e.
$$
    \val(p)=f(p,\IEmap(p)),\quad p\in\dom\IEmap.
$$
It should be noticed that $\val:\dom\IEmap\longrightarrow\R^m$ is a single-valued mapping, even
for those $p$ for which the set $\IEmap(p)$ contains more than one
element. Indeed, the fact that $\Cmap(p)$ is pointed (according to $(\mathfrak{A}_4)$)
entails that $f(p,x)$ takes the same value in $\R^m$ for every $x\in\IEmap(p)$.
This is a remarkable feature distinguishing ideal efficiency from
the mere efficiency. Indeed, for the latter type of solution
the values of the objective function give raise to the so-called efficient
front mapping, which is generally a set-valued mapping.

In order to put the qualitative analysis of problems $(\VOP_p)$ in the framework
of the theory exposed in Section \ref{Sect:3}, let us introduce the set-valued mapping
$\VOPmap:P\times\X\longrightarrow\Y$ defined as follows
$$
  \VOPmap(p,x)=f(p,\Rmap(p))-f(p,x).
$$

Given a closed, convex cone $\{\nullv\}\ne C\subsetneqq\R^m$, a mapping
$g:\R^n\longrightarrow\R^m$ is said to be $C$-decreasing at $x_0\in\R^n$
if the mapping $-g$ is $C$-increasing at the same point. In such an event, the
value
$$
   \dec{g}{C}{x_0}=\inc{-g}{C}{x_0}.
$$
is called {\it exact bound of $C$-decrease} of $g$ at $x_0$. In other
terms, $g$ is $C$-decreasing at $x_0$ if there exist $\delta>0$ and
$\alpha>1$ such that
\begin{equation}    \label{in:defCdecr}
  \forall r\in (0,\delta]\ \exists u\in\ball{x_0}{r}:\
  \ball{g(u)}{\alpha r}\subseteq \ball{g(x_0)-C}{r}.
\end{equation}

On the base of the solvability results achieved in Section \ref{Sect:3},
the next proposition provide sufficient conditions for the existence
of ideal efficient solutions and a related error bound, in the presence of global variations
of the parameter, which affect all problem data (including the partial order).
A key role in its formulation is played by the
following constant:
$$
  \underline{\alpha}_{f,\Cmap,\Rmap}=\inf\bigl\{
  \dec{f(p,\cdot)}{\Cmap(p)}{x}:\ (p,x)\in P\times\R^n,\
  \VOPmap(p,x)\not\subseteq\Cmap(p),\ x\in\Rmap(p)\bigl\}.
$$

\begin{proposition}{\bf (Global ideal solvability)}     \label{pro:gloisolv}
With reference to a family of problems $(\VOP_p)$, suppose that:
\begin{itemize}

\item[(i)] for every $p\in P$, the set $f(p,\Rmap(p))$ is closed
and $\Cmap(p)$-bounded;

\item[(ii)] it is $\underline{\alpha}_{f,\Cmap,\Rmap}>1$;

\item[(iii)] $f(p,\cdot)$ is continuous on $\R^n$ for
every $p\in P$ and
Lipschitz continuous on $\R^n\backslash\Rmap(p)$, with a constant
$0<\ell_f<\underline{\alpha}_{f,\Cmap,\Rmap}-1$, uniform in $p\in P$.

\end{itemize}
Then, it is $\IEmap(p)\ne\varnothing$ for every $p\in P$, and for every
$\alpha\in\left({\underline{\alpha}_{f,\Cmap,\Rmap}-\ell_f+1\over 2},
{\underline{\alpha}_{f,\Cmap,\Rmap}-\ell_f}\right)$
and $x_0\in\R^n$ it holds
\begin{eqnarray}     \label{in:condlfalf1}
   \dist{x_0}{\IEmap(p)} &\le & {\exc{f(p,\Rmap(p)-f(p,x_0)}{\Cmap(p)}+
   (\underline{\alpha}_{f,\Cmap,\Rmap}-\alpha)\dist{x_0}{\Rmap(p)}\over
   \underline{\alpha}_{f,\Cmap,\Rmap}-\alpha-\ell_f}, \nonumber \\
   & &\quad\forall p\in P.
\end{eqnarray}
\end{proposition}

\begin{proof}
All the assertions in the thesis follow at once by applying
Proposition \ref{pro:gloconsolv}. This can be done under
the current hypotheses. Indeed, as for every $p\in P$ the set
$f(p,\Rmap(p))\backslash\Cmap(p)$ is bounded, then it is clear
that $\exc{f(p,\Rmap(p))\backslash\Cmap(p)}{\Cmap(p)}<+\infty$
and hence, for every fixed $x\in\R^n$, one has
\begin{eqnarray*}
  \exc{\VOPmap(p,x)}{\Cmap(p)} &=& \exc{f(p,\Rmap(p))-f(p,x)}{\Cmap(p)} \\
   &\le& \exc{f(p,\Rmap(p))-f(p,x)}{f(p,\Rmap(p))}    \\
   & &  +\exc{f(p,\Rmap(p))}{\Cmap(p)} \\
   &\le& \|f(p,x)\|+\exc{f(p,\Rmap(p))\backslash\Cmap(p)}{\Cmap(p)}<+\infty.
\end{eqnarray*}
Thus, hypothesis (i) of Proposition \ref{pro:gloconsolv} is fulfilled.

Now, let $p\in P$ be fixed. By virtue of hypothesis (ii), the mapping
$-f(p,\cdot)$ is $\Cmap(p)$-increasing at each $x\in\R^n\backslash\Rmap(p)$
with exact bound $\inc{-f(p,\cdot)}{\Cmap(p)}{x}\ge \underline{\alpha}_{f,\Cmap,\Rmap}$.
Since the set-valued mapping $x\rightrightarrows f(p,\Rmap(p))$ is
constant and therefore Lipschitz continuous with constant $\ell=0$,
then, according to Proposition \ref{pro:Cincrprosta}, the set-valued
mapping $\VOPmap(p,\cdot)=f(p,\Rmap(p))-f(p,\cdot)$ is $\Cmap(p)$-increasing
with exact bound
$$
  \inc{\VOPmap(p,\cdot)}{\Cmap(p)}{x}\ge
  \underline{\alpha}_{f,\Cmap,\Rmap},\quad\forall
  x\in\Rmap(p).
$$
By consequence one obtains
\begin{eqnarray*}
  \widetilde{\alpha}_{\VOPmap,\Cmap,\Rmap} &=& \inf\bigg\{\inc{\VOPmap(p,\cdot)}{\Cmap(p)}{x}:\
  (p,x)\in P\times\R^n,\   \\
  & & \VOPmap(p,x)\not\subseteq\Cmap(p),\ x\in\Rmap(p)\bigg\}
  \ge \underline{\alpha}_{f,\Cmap,\Rmap}>1,
\end{eqnarray*}
which ascertains the fulfilment of hypothesis (ii).

Since $f(p,\cdot)$ is continuous on $\R^n$ and Lipschitz continuous on $\R^n\backslash\Rmap(p)$,
with a uniform constant $\ell_f<\underline{\alpha}_{f,\Cmap,\Rmap}-1$,
$\VOPmap(p,\cdot)$ is l.s.c. on $\R^n$ and Lipschitz continuous on $\R^n\backslash\Rmap(p)$,
by virtue of the invariance under
translation of the Euclidean distance. Thus, also hypothesis (iii)
of Proposition \ref{pro:gloconsolv} happens to be satisfied.

As $\Rmap$ takes nonempty, closed values and every nonempty closed
subset of $\R^n$ is known to be proximinal, Proposition \ref{pro:gloconsolv}
can be actually applied. This completes the proof.
\end{proof}

The example below demonstrates the fact that, in general, the condition
in hypothesis (iii) of Proposition \ref{pro:gloisolv}, linking $\ell_f$
and $\underline{\alpha}_{f,\Cmap,\Rmap}$, is essential to ensure the
global existence of ideal efficient solutions to $(\VOP_p)$.

\begin{example}
Consider the specific instance of $(\VOP_p)$ discussed in \cite[Example 1]{Uder23},
whose data are: $P=[0,2\pi]$, $n=m=2$, $C=\R^2_+$, $f:[0,2\pi]\times\R^2
\longrightarrow\R^2$ given by
$$
   f(p,x)=O_px,\quad\hbox{ with }\quad O_p=
   \left(\begin{array}{cc}
            \cos p & -\sin p \\
            \sin p & \cos p
          \end{array}\right)\in\Rot(2),
$$
and $\Rmap:[0,2\pi]\rightrightarrows\R^2$ constant, namely $\Rmap(p)=T$,
where $T=\{x=(x_1,x_2)\in\R^2:\ x_1\ge 0,\ x_2\ge 0,\ x_1+x_2\le 1\}$.
As the associated solution mapping $\IEmap:[0,2\pi]\rightrightarrows\R^2$
results in
$$
  \IEmap(p)=\left\{\begin{array}{cl} \{(0,0)\} & \quad\hbox{ if } p\in\{0,\, 2\pi\}, \\
                                     \\
                                \{(1,0)\} & \quad\hbox{ if } p\in \left[{\pi\over 2},{3\over 4}\pi\right], \\
                                      \\
                                \{(0,1)\} & \quad\hbox{ if } p\in \left[{5\over 4}\pi,{3\over 2}\pi\right], \\
                                      \\
                                \varnothing & \quad\hbox{ otherwise,}
                   \end{array}\right.
$$
one sees that such a parameterized family of  $(\VOP_p)$ lacks
of global ideal solvability.
On the other hand, the set $f(p,T)=O_p(T)$ is trivially closed and bounded.
Moreover, on the base of what discussed in Example \ref{ex:rescarot},
one has
$$
  \inc{-f(p,\cdot)}{\R^2_+}{x}=\frac{1}{\sqrt{2}}+1,\quad\forall
  x\in\R^2,
$$
and therefore, in this particular case,
$$
   \underline{\alpha}_{f,\R^2_+,\Rmap}=
   \inf_{(p,x)\in[0,2\pi]\times\R^2}\dec{f(p,\cdot)}{\R^2_+}{x}
   =\frac{1}{\sqrt{2}}+1.
$$
As a linear mapping, $f(p,\cdot)$ is automatically continuous on $\R^n$.
Nonetheless, as the smallest Lipschitz constant $\ell_f$ of $f(p,\cdot)$,
uniformly in $[0,2\pi]$, is $\ell_f=1$ (remember that $O_p\Usfer=\Usfer$),
one finds
$$
   \underline{\alpha}_{f,\R^2_+,\Rmap}-1=
   \frac{1}{\sqrt{2}}<1=\ell_f,
$$
so that condition in hypothesis (iii) of Proposition \ref{pro:gloisolv}
fails to hold.
\end{example}

The next lemma is ancillary to the proof of the implicit function theorem
related to $(\VOP_p)$.

\begin{lemma}    \label{lem:fCconvFCconc}
Given a family of problems $(\VOP_p)$, let $f(p,\cdot)$ be $\Cmap(p)$-convex on $\R^n$,
for any $p\in P$. Then, $\VOPmap(p,\cdot):\R^n\rightrightarrows\R^m$ is
$\Cmap(p)$-concave in $\R^n$.
\end{lemma}

\begin{proof}
Take arbitrary $x_1,\, x_2\in\R^n$ and $t\in [0,1]$. By virtue of
the $\Cmap(p)$-convexity of $f(p,\cdot)$ one has
$$
  -f(p,tx_1+(1-t)x_2)\in -tf(p,x_1)-(1-t)f(p,x_2)+\Cmap(p).
$$
Besides, as $p\in P$ is fixed, it holds
$$
  f(p,\Rmap(p))\subseteq t f(p,\Rmap(p))+(1-t)f(p,\Rmap(p)).
$$
By combining the above two inclusions, one obtains
\begin{eqnarray*}
  \VOPmap(p,tx_1+(1-t)x_2) &=& f(p,\Rmap(p))-f(p,tx_1+(1-t)x_2) \\
  &\subseteq & t f(p,\Rmap(p))+(1-t)f(p,\Rmap(p))  \\
  & & -tf(p,x_1)-(1-t)f(p,x_2)+\Cmap(p) \\
  &=& t\VOPmap(p,x_1)+(1-t)\VOPmap(p,x_2)+\Cmap(p),
\end{eqnarray*}
which completes the proof.
\end{proof}

The next theorem provides conditions able to ensure the existence
of ideal efficient solutions to $(\VOP_p)$ for perturbations of $p$ all over $P$
as well as their continuous dependence on $p\in P$. This
result is established in the particular case of $\Cmap$ being
constantly $C$.

\begin{theorem}{\bf (Global continuous dependence)}    \label{thm:glocontief}
With reference to a parameterized family of problems $(\VOP_p)$,
suppose that:

\begin{itemize}

\item[(i)] $(P,\tau)$ is a paracompact topological space;

\item[(ii)] for every $p\in P$, the set $f(p,\Rmap(p))$ is closed
and $C$-bounded;

\item[(iii)] it is $\underline{\alpha}_{f,C,\Rmap}>1$;

\item[(iv)] $f(p,\cdot)$ is continuous on $\R^n$ for every $p\in P$ and
Lipschitz continuous on $\R^n\backslash\Rmap(p)$, with a constant
$\ell_f$ such that $0<\ell_f<\underline{\alpha}_{f,C,\Rmap}-1$,
uniform on $P$;

\item[(v)] $f(p,\cdot)$ is $C$-convex on $\R^n$, for every $p\in P$;

\item[(vi)] $f(\cdot,x)$ is continuous with respect to $p$ on $P$,
uniformly in $x\in\R^n$;

\item[(vii)] $\Rmap:P\rightrightarrows\R^n$ is continuous on $P$ and
takes convex values.

\end{itemize}
Then, $\dom\IEmap=P$ and there exists a function $\ief:P\longrightarrow\R^n$,
which is continuous on $P$ and such that $\ief(p)\in\IEmap(p)$,
for every $p\in P$.
\end{theorem}

\begin{proof}
The proof consists in showing that, in this particular context of
application, one can employ Theorem \ref{thm:congloimpfun} with
$\Fmap=\VOPmap$, so that $\CSmap=\IEmap$ and $\cimpf=\ief$.
From the proof of Proposition \ref{pro:gloisolv}, it is clear that
hypotheses (i)--(iv) of Theorem \ref{thm:congloimpfun} are satisfied
as a consequence of the current hypotheses (i)--(iv).
On account of hypothesis (v), Lemma \ref{lem:fCconvFCconc} implies
that $\VOPmap(p,\cdot)$ is $C$-concave on $\R^n$, for every $p\in P$.
The current hypothesis (vii) entails, in particular, hypothesis (vii)
of Theorem \ref{thm:congloimpfun}.

It remains to show that, upon the current hypotheses, $\VOPmap(\cdot,x):
P\rightrightarrows\R^m$ is Hausdorff $C$-u.s.c. on $P$. To do so,
take an arbitrary $p_0\in P$ and fix $\epsilon>0$. As a consequence
of hypothesis (vi), there exists a neighbourhood $V_\epsilon$ of
$p_0$ such that
$$
  \|f(p,x_1)-f(p_0,x_1)\|\le\frac{\epsilon}{4},\quad\forall
  p\in V_\epsilon, \forall x_1\in\R^n.
$$
Thus, from hypothesis (iv) it follows
\begin{eqnarray*}
  \|f(p,x_1)-f(p_0,x_2)\| &\le& \|f(p,x_1)-f(p_0,x_1)\|+\|f(p_0,x_1)-f(p_0,x_2)\| \\
   &\le& \frac{\epsilon}{4}+\ell_f\|x_1-x_2\|,\quad
   \forall x_1,\, x_2\in\R^n,\quad\forall p\in V_\epsilon.
\end{eqnarray*}
This inequality means
$$
   f(p,x_1)\in f(p_0,x_2)+\left(\frac{\epsilon}{4}+\ell_f\|x_1-x_2\|\right)\Uball,
   \quad \forall x_1,\, x_2\in\R^n,\quad\forall p\in V_\epsilon,
$$
which in turn implies
$$
   f(p,x_1)\in f(p_0,\Rmap(p_0))+\left(\frac{\epsilon}{4}+\ell_f\|x_1-x_2\|\right)\Uball,
   \quad \forall x_1\in\R^n,\ \forall x_2\in\Rmap(p_0),\ \forall p\in V_\epsilon.
$$
As the last inclusion is true for every $x_2\in\Rmap(p_0)$, it is
possible to write
\begin{eqnarray*}
 f(p,x_1) &\in & \bigcap_{x_2\in\Rmap(p_0)} \left[f(p_0,\Rmap(p_0))+
 \left(\frac{\epsilon}{4}+\ell_f\|x_1-x_2\|\right)\Uball\right]  \\
  &\subseteq& f(p_0,\Rmap(p_0))+\left(\frac{\epsilon}{3}+\ell_f\cdot
  \inf_{x_2\in\Rmap(p_0)} \|x_1-x_2\|\right)\Uball  \\
  &=&  f(p_0,\Rmap(p_0))+\left(\frac{\epsilon}{3}+\ell_f
  \dist{x_1}{\Rmap(p_0)}\right)\Uball,\quad
  \forall x_1\in\R^n,\ \forall p\in V_\epsilon.
\end{eqnarray*}
Now, as $\Rmap$ is, in particular, u.s.c. at $p_0$, there exists a
neighbourhood $W_\epsilon$ of $p_0$ such that
$$
  \Rmap(p)\subseteq\inte\ball{\Rmap(p_0)}{{\epsilon\over 2\ell_f}},
 \quad\forall p\in W_\epsilon.
$$
Thus, from the previous inclusions one obtains
\begin{eqnarray*}
  f(p,\Rmap(p)) &\subseteq& \bigcup_{x_1\in\ball{\Rmap(p_0)}{\epsilon/2\ell_f}}
  \left[f(p_0,\Rmap(p_0))+\left(\frac{\epsilon}{3}+\ell_f
  \dist{x_1}{\Rmap(p_0)}\right)\Uball\right] \\
     &=& f(p_0,\Rmap(p_0))+\sup_{x_1\in\ball{\Rmap(p_0)}{\epsilon/2\ell_f}}
   \left(\frac{\epsilon}{3}+\ell_f  \dist{x_1}{\Rmap(p_0)}\right)\Uball  \\
   &\subseteq & f(p_0,\Rmap(p_0))+\left(\frac{\epsilon}{3}+\ell_f\cdot
   \frac{\epsilon}{2\ell_f}\right)\Uball  \\
   &\subseteq& \ball{f(p_0,\Rmap(p_0))}{\epsilon}+C
   \subseteq \ball{f(p_0,\Rmap(p_0))+C}{\epsilon} ,\quad\forall
   p\in V_\epsilon\cap W_\epsilon.
\end{eqnarray*}
This inclusion shows that the set-valued mapping $p\rightrightarrows
f(p,\Rmap(p))$ is Hausdorff $C$-u.s.c. at $p_0$ and therefore so is
also its translation $\VOPmap(\cdot,x)=f(\cdot,\Rmap(\cdot))-f(\cdot,x)$ by the
continuity of $f(\cdot,x)$, for every fixed $x\in\R^n$.
Thus, one can deduce that also hypothesis (vi) of Theorem \ref{thm:congloimpfun}
is fulfilled, thereby completing the proof.
\end{proof}

\begin{remark}
It is well known that, as in the scalar case, the property of
$C$-convexity for vector-valued functions may imply Lipschitz
continuity. Several results in this sense can be found in \cite{AhTaCo16}.
Nevertheless, since hypothesis (iv) in Theorem \ref{thm:glocontief}
refers to Lipschitz continuity with the same constant on the
whole set $\R^n\backslash\Rmap(p)$, whereas all the results in
\cite{AhTaCo16} enable to achieve only local Lipschitz continuity,
such hypothesis can not be dropped out.
\end{remark}

As a remarkable consequence of the continuous dependence of ideal
efficient solutions to $(\VOP_p)$ on $p\in P$, one obtains the global
continuity behaviour of the ideal value mapping associated with
$(\VOP_p)$.

\begin{corollary}{\bf (Continuity of the ideal value)}
Given a parameterized family of problems $(\VOP_p)$, suppose that
all the hypotheses of Theorem \ref{thm:glocontief} are satisfied. Then, the
ideal value function $\val:P\longrightarrow\R^m$ is continuous
on $P$.
\end{corollary}

\begin{proof}
According to the definition of ideal value mapping, it is $\val(p)
=f(p,x_p)$, with $x_p\in\IEmap(p)$, for every $p\in P$. Thus, let
$\ief:P\longrightarrow\R^n$ be a continuous selection of $\IEmap$,
whose existence is guaranteed by Theorem \ref{thm:glocontief}. Then, function
$\val:P\longrightarrow\R^m$ can be well defined by
$$
   \val(p)=f(p,\ief(p)),\quad\forall p\in P.
$$
It remains to prove that, via the above definition, $\val$ is continuous at each point of
$P$. So, take an arbitrary $\bar p\in P$ and fix an arbitrary $\epsilon>0$.
Since $f(\cdot,\ief(\bar p))$ is continuous at $\bar p$ owing to
hypothesis (vi), there exists a neighbourhood $U^1_\epsilon$ of
$\bar p$ such that
\begin{equation}    \label{in:valcon1}
  \|f(p,\ief(\bar p))-f(\bar p,\ief(\bar p))\|\le\frac{\epsilon}{2},
  \quad\forall p\in U^1_\epsilon.
\end{equation}
Since $\ief$ is continuous in particular at $\bar p$, there exists
a neighbourhood $U^2_\epsilon$ of $\bar p$ such that
\begin{equation}    \label{in:valcon2}
  \|\ief(p)-\ief(\bar p)\|\le\frac{\epsilon}{2\ell_f},
  \quad\forall p\in U^2_\epsilon.
\end{equation}
Thus, by setting $U_\epsilon=U^1_\epsilon\cap U^2_\epsilon$
and by taking into account hypothesis (iv) of Theorem \ref{thm:glocontief},
from inequalities (\ref{in:valcon1}) and (\ref{in:valcon2})
one obtains
\begin{eqnarray*}
  \|\val(p)-\val(\bar p)\| &=& \|f(p,\ief(p))-f(\bar p,\ief(\bar p))\| \\
   &\le &  \|f(p,\ief(p))-f(p,\ief(\bar p))\|+\|f(p,\ief(\bar p))-f(\bar p,\ief(\bar p))\|  \\
   &\le & \ell_f\|\ief(p)-\ief(\bar p))\|+\frac{\epsilon}{2}\le\epsilon,
   \quad\forall p\in U_\epsilon.
\end{eqnarray*}
By arbitrariness of $\epsilon$ and $\bar p\in P$, the above inequality
completes the proof.
\end{proof}

The next example aims at illustrating the role of the $C$-decrease
property in the topic at the issue, through an elementary situation.

\begin{example}
Consider the class of $(\VOP_p)$ problems defined by the following data:
$P=\R$ (equipped with its usual Euclidean structure), $n=1$, $m=2$, $C=\R^2_+$,
$f:\R\times\R\longrightarrow\R^2$ given by
$$
  f(p,x)=\binom{|x-\varphi(p)|}{|x-\varphi(p)|},
$$
and $\Rmap:\R\rightrightarrows\R$ given by $\Rmap(p)=\R$ (so the
problem is unconstrained) for every $p\in\R$, where $\varphi:\R\longrightarrow\R$ is
any continuous functions on $\R$.
With the above data, it is readily seen that $\IEmap:\R\rightrightarrows\R$
becomes $\IEmap(p)=\{\varphi(p)\}$ and hence $\ief:\R\longrightarrow\R$
amounts to $\ief(p)=\varphi(p)$, while $\val:\R\longrightarrow\R^2$
vanishes everywhere, namely $\val(p)=\nullv$ for every $p\in\R$.
As it is so plane, this instance of $(\VOP_p)$ can be easily checked
to fall in the range of application of Theorem \ref{thm:glocontief}.
Indeed, as a metric space $(\R,|\cdot|)$ is paracompact. Fixed $p\in\R$,
it is $f(p,\R)=\{y=(y_1,y_2)\in\R^2:\ y_1=y_2\ge 0\}$, which is closed
and $\R^2_+$-bounded, as $f(p,\R)\backslash\R^2_+=\varnothing$.
Hypothesis (iv) is satisfied because $f(p,\cdot)$ is continuous on $\R$,
while $\R\backslash\Rmap(p)=\varnothing$, so no Lipschitz continuity
with uniform constant $\ell_f$ is required.
Since each of its component is convex on $\R$, $f(p,\cdot):\R\longrightarrow\R^2$
turns out to be $\R^2_+$-convex on $\R$, for every $p\in\R$.
Hypothesis (vi) is satisfied by virtue of the continuity of $\varphi$,
on account of the definition of $f$.
Since the fulfilment of hypothesis (vii) is self evident, what
remains to show is that $\underline{\alpha}_{f,C,\Rmap}>1$.
So, fix $p\in\R$ and take an arbitrary $x\in\R\backslash\{\varphi(p)\}$.
Letting $\delta=\frac{1}{2}|x-\varphi(p)|$, take an arbitrary $ r\in
(0,\delta]$.

If $x>\varphi(p)$, then set $u=x-r\in\ball{x}{r}$ and observe that
$u>\varphi(p)$, because $x-r\ge x-\frac{1}{2}|x-\varphi(p)|=
\frac{1}{2}(x+\varphi(p))>\varphi(p)$. Thus, one obtains
\begin{eqnarray*}
  \ball{f(p,u)}{2r} &=& \binom{x-r-\varphi(p)}{x-r-\varphi(p)}+2r\Uball \\
    &\subseteq & \binom{x-\varphi(p)}{x-\varphi(p)}+r\Uball-\R^2_+
    =\ball{f(p,x)-\R^2_+}{r}.
\end{eqnarray*}
Similarly, if $x<\varphi(p)$, then set $u=x+r\in\ball{x}{r}$ and observe that
$u<\varphi(p)$, because $x+r\le x+\frac{1}{2}|x-\varphi(p)|=
\frac{1}{2}(x+\varphi(p))<\varphi(p)$. Consequently, one obtains
\begin{eqnarray*}
  \ball{f(p,u)}{2r} &=& \binom{\varphi(p)-x-r}{\varphi(p)-x-r}+2r\Uball \\
    &\subseteq & \binom{\varphi(p)-x}{\varphi(p)-x}+r\Uball-\R^2_+
    =\ball{f(p,x)-\R^2_+}{r}.
\end{eqnarray*}
In both the cases the inclusion in (\ref{in:defCdecr}) holds true.
Thus the above inclusions show that
$$
  \dec{f(p,\cdot)}{\R^2_+}{x}\ge 2,\quad\forall x\in\R\backslash\{\varphi(p)\},
$$
which leads to the estimate
$$
  \underline{\alpha}_{f,C,\Rmap}\ge 2>1.
$$
So, also hypothesis (iii) of Theorem \ref{thm:glocontief} happens to
be satisfied. In contrast, it is worth observing that
$$
  \dec{f(p,\cdot)}{\R^2_+}{\varphi(p)}=1.
$$
\end{example}

\vskip1cm




\begin{thebibliography}{99}


\bibitem{AhTaCo16} Anh Tuan, V., Tammer, C., Zălinescu, C.:
{The Lipschitzianity of convex vector and set-valued functions}.
TOP 2016; \textbf{24}(1): 273--299.

\bibitem{Arut15} Arutyunov A.V.:
{Caristi's Condition and Existence of a Minimum of a Lower Bounded
Function in a Metric Space. Applications to the Theory of Coincidence
Points}.
Proc. Steklov Inst. Math. 2015; \textbf{291}(1):24--37.

\bibitem{AruZhu20} Arutyunov A.V., Izmailov A.F., Zhukovskiy S.E.:
{Continuous selections of solutions for locally Lipschitzian equations }.
J. Optim. Theory Appl. 2020; \textbf{185}(3):679--699.

\bibitem{ArZhMo23} Arutyunov A.V., Zhukovskiy S.E., Mordukhovich B.Sh.:
{Implicit Function Theorems for Continuous Mappings and Their Applications}.
Math. Notes 2023; \textbf{113}(6):749--759.

\bibitem{Bedn95} Bednarczuk E.M.:
{Berge-type theorems for vector optimization problems}. Optimization
1995; \textbf{32}(4):373–-384.

\bibitem{BenNem98} Ben-Tal A., Nemirovski A.:
{Robust convex optimization}, Math. Oper. Res. 1998; \textbf{23}:769--805.

\bibitem{BeGhNe09} Ben-Tal A., Ghaoui L.E., Nemirovski A.:
{Robust optimization}. Princeton: Princeton University Press; 2009.

\bibitem{CaLoMoPa20} C\'anovas M.J., L\'opez M.A., Mordukhovich B.S., Parra J.:
{Subdifferentials and Stability Analysis of Feasible Set and Pareto Front
Mappings in Linear Multiobjective Optimization}.
Vietnam J. Math. 2020; \textbf{48}:315--334.

\bibitem{CraLuu00} Craven B.D., Luu D.V.:
{Perturbing convex multiobjective programs}. Optimization 2000;
\textbf{48}(4):391–-407.

\bibitem{Flor02} Flores-Baz\'an F.:
{Ideal, weakly efficient solutions for vector optimization problems}.
Math. Program. Ser. A 2002; \textbf{93}:543--475.

\bibitem{GraDug03} Granas A., Dugundji J.:
{Fixed Point Theory}. New-York: Springer; 2003.

\bibitem{HuMoYa08} Huy N.Q., Mordukhovich B.S., Yao J.C.:
{Coderivatives of frontier and solution maps in parametric multiobjective
optimization}. Taiwanese J. Math. 12 (2008); \textbf{12}(8):2083--2111.

\bibitem{Jahn11} Jahn J.:
{Vector Optimization. Theory, Applications, and Extensions}.
Berlin: Springer-Verlag; 2011.

\bibitem{Luc89} Luc D.T.:
{Theory of Vector Optimization}. Berlin: Springer; 1989.

\bibitem{Mich56} Michael E.:
{Continuous selections I}. Ann. of Math. (2) 1956; \textbf{63}:361--382.

\bibitem{Mord06} Mordukhovich B. S.:
{Variational analysis and generalized
differentiation I. Basic theory}. Berlin: Springer-Verlag; 2006.

\bibitem{Peno13} Penot J.-P.:
{Calculus Without Derivatives}. New-York: Springer; 2013.

\bibitem{SaNaTa85} Sawaragi Y., Nakayama H., Tanino T.:
{Theory of multiobjective optimization}. New York: Academic press;
1985.

\bibitem{SteSee78} Steen L.A., Seebach J.A.:
{Counterexamples in topology}. New-York: Springer-Verlag; 1978.

\bibitem{Tani90} Tanino T.:
{Stability and sensitivity analysis in multiobjective nonlinear programming}.
Ann. Oper. Res. 1990; \textbf{27}(1-4):97–-114.

\bibitem{Thib23} Thibault L.:
{Unilateral variational analysis in Banach spaces. Part I-general
theory}. Singapore: World Scientific Publishing Co. Pte. Ltd; 2023.

\bibitem{Uder19} Uderzo A.:
{On some generalized equations with metrically $C$-increasing mappings:
solvability and error bounds with applications to optimization}. Optimization 2019;
\textbf{68}(1):227--253.

\bibitem{Uder21} Uderzo A.:
{On the Quantitative Solution Stability of Parameterized Set-Valued
Inclusions}. Set-Valued Var. Anal. 2021; \textbf{29}:425--451.

\bibitem{Uder21b} Uderzo A.:
{On differential properties of multifunctions defined implicitly by
set-valued inclusions}. Pure Appl. Funct. Anal. 2021; \textbf{6}(6):1509--1531.

\bibitem{Uder22} Uderzo A.:
{On tangential approximations of the solution set of set-valued
inclusions}. J. Appl. Anal. 2022; \textbf{28}(1):11--33.

\bibitem{Uder23} Uderzo A.:
{Conditions for the stability of ideal efficient solutions in parametric vector
optimization via set-valued inclusions}, J. Global Optim. 2023; \textbf{85}:917--940.

\bibitem{Uder23b} Uderzo A.:
{Some enhanced existence results for strong vector equilibrium problems}.
Pure Appl. Funct. Anal. 2023; \textbf{8}(3):987--1011.

\end{thebibliography}
\end{document}